\documentclass{amsart}
\usepackage{amssymb, a4wide, mathdots, url,hyperref,graphicx}
\usepackage[usenames]{color}
\usepackage{enumerate}
\usepackage{wasysym}
\usepackage{dsfont}
\usepackage{comment}
\usepackage{MnSymbol}
\usepackage{eqnarray, amsmath}
\usepackage{tikz}[2013/12/13]
\usepackage{tikzit}
\usepackage{tikz-cd}
\usepackage[normalem]{ulem}

\tikzstyle{dot}=[fill=black, draw=black, shape=circle]

\tikzstyle{arrow}=[->]
\tikzstyle{axis}=[<->, draw=black]
\tikzstyle{axis red}=[draw=red, <->]
\tikzstyle{new edge style 0}=[-, draw={rgb,255: red,191; green,191; blue,191}]
\tikzstyle{arrow red}=[draw=red, ->]
\tikzstyle{line red}=[-, draw=red]

\newcommand{\ww}[1]{{\color{black} #1}}
\newcommand{\www}[1]{{\color{blue} #1}}
\newcommand{\sgn}{\mathrm{sgn}}

\newcommand{\Z}{\mathbb Z}
\newcommand{\Ha}{\mathbb H}

\newcommand{\Ca}{\mathbb C}
\newcommand{\C}{\mathbb C}
\newcommand{\R}{\mathbb R}

\newcommand{\inte}{\mathrm{int}}

\newcommand{\SL}{\operatorname{SL}}

\newtheorem{theorem}{Theorem}[section]
\newtheorem*{theorem*}{Theorem}
\newtheorem{lemma}[theorem]{Lemma}

\newtheorem{proposition}[theorem]{Proposition}

\theoremstyle{remark}
\newtheorem{remark}[theorem]{Remark}

\title{On the real zeros of depth 1 quasimodular forms}
\author{Bo-Hae Im}
\address{Department of Mathematical sciences,  KAIST, 291 Daehak-ro, Yuseong-gu, Daejeon 34141, South Korea}
\email{bhim@kaist.ac.kr}

\author{Wonwoong Lee}
\address{Department of Mathematics, The University of Hong Kong, Pokfulam, Hong Kong}
\email{leeww041@hku.hk}

\thanks{Bo-Hae Im was supported by Basic Science Research Program through the National Research Foundation of Korea(NRF) funded by the Korea government(MSIT)(NRF-2023R1A2C1002385).}

\date{\today}

\begin{document}

\maketitle

\begin{abstract}
    We discuss the critical points of modular forms, or more generally the zeros of quasimodular forms of depth $1$ for $\mathrm{PSL}_2(\mathbb Z)$. In particular, we consider the derivatives of the unique weight $k$ modular forms $f_k$ with the maximal number of consecutive zero Fourier coefficients following the constant $1$. Our main results state that (1) every zero of a depth $1$ quasimodular form near the derivative of the Eisenstein series in the standard fundamental domain lies on the geodesic segment $\{z \in \mathbb H: \Re(z)=1/2\}$, and (2)  more than quarter of zeros of $f_k$ in the standard fundamental domain lie on the geodesic segment $\{z \in \mathbb H: \Re(z)=1/2\}$ for large enough $k$ with $k\equiv 0 \pmod{12}$.
\end{abstract}

\section{Introduction}

Holomorphic modular forms are complex-analytic functions on the upper half-plane that satisfy specific transformation properties under the action of a congruence subgroup of $\mathrm{SL}_2(\Z)$
 and exhibit suitable behavior at the cusps. They form a central object of study in number theory. For background on the classical theory, we refer the reader to \cite{DS05} and \cite{Ser73}.

In contrast to modular forms, which satisfy strict invariance under the action of a congruence subgroup of $\mathrm{SL}_2(\Z)$, quasimodular forms allow a controlled failure of modularity, captured by a finite expansion in terms of a non-modular correction term. More precisely,
let $\Ha$ be the upper half-plane of complex numbers, $k$ be a non-negative even integer, and $p$ be a non-negative integer. A quasimodular form of weight $k$ and depth $p$ for the full modular group $\Gamma:=\mathrm{PSL}_2(\Z)$ is a holomorphic function $f:\Ha \to \Ca$ satisfying the following conditions:
\begin{enumerate}[\normalfont(i)]
    \item There exist holomorphic functions \(Q_i(f)\) on $\Ha$ for \(i=0,1,\ldots,p\) that satisfy
\begin{equation*}
    f| _k \gamma=\sum_{i=0}^{p} Q_i(f)X(\gamma)^i, \quad \text{ with }  Q_{p}(f) \not\equiv 0, \text{ for all } \gamma=\begin{pmatrix}
a & b \\
c & d
\end{pmatrix} \in \Gamma,
\end{equation*}
 where the operator $| _k\gamma$ is defined by
\begin{equation*}
    f|_k\gamma(z)=(cz+d)^{-k}f(\gamma z) \quad \text{for } z \in \Ha,
\end{equation*}
and the function \(X(\gamma)\) is defined by
\begin{equation*}
    X(\gamma)(z)=\frac{c}{cz+d} \quad \text{for } z\in \Ha.
\end{equation*}
    \item \(f\) is polynomially bounded, i.e., there exists a constant \(\alpha>0\) such that $$f(z)=O((1+|  z|  ^2)/y)^{\alpha},$$ as \(y \to \infty\) and \(y \to 0\), where $z=x+iy$ with $x,y \in \R$.
\end{enumerate}
Instead of the second condition, one may replace it with {\it holomorphic at cusps} for the functions $Q_i(f)$, which is further discussed in \cite{Roy12}. By convention, we consider the $0$ function as a quasimodular form. Throughout this paper,  for the sake of brevity we have omitted explicit mention,  but it should be noted that all modular and quasimodular forms discussed herein are defined for $\Gamma$. The space of quasimodular forms of weight $k$ and depth $\leq p$ is denoted by $\widetilde{M_k}^{(\leq p)}$, and $M_k$ simply represents $\widetilde{M_k}^{(\leq 0)}$, the space of modular forms of weight $k$.

The Eisenstein series $E_2$ of weight $2$ and the derivatives of modular forms are standard examples of quasimodular forms. More precisely, if $f$ is a quasimodular form of weight $k$ and depth $p$, then $f'$ is either the zero or a quasimodular form of weight $k+2$ and depth $p+1$. In particular, the derivative of a non-constant holomorphic modular form is a depth $1$ quasimodular form.

There have been various and extensive studies on the zeros of modular and quasimodular forms in the standard fundamental domain $F$ of $\Gamma$, where
\begin{align*}
    F=\{z \in \Ha: -1/2<\Re(z)\leq 1/2, \quad |  z|  >1 \text{ if } \Re(z)<0, \quad |  z|  \geq 1 \text{ if } \Re(z)\geq0\} \cup \{\infty\}.  
\end{align*}
Rankin and Swinnerton-Dyer's celebrated result \cite{RS70} states that if $k\geq 4$ is an even integer, every zero of the Eisenstein series $E_k$ for $\Gamma$ lies on the unit circle. Getz \cite{Get04} proved that the same property holds for certain class of functions; modular forms $f$ `near' the Eisenstein series, those are, $f=E_k+\sum a_i E_{k-12i}\Delta^i$ for $a_i \in \R$ with small enough $|  a_i|  $. 

Let us introduce the so-called {\it Gap function} which can be viewed as a standard basis element of the space of weakly holomorphic modular forms. These are of the form
\begin{align*}
    f_{k,m}(z)=q^{-m}+O(q^{\ell+1})
\end{align*}
for $m \geq -\ell$, where $q:=e^{2\pi i z}$ and $\ell$ is the dimension of the space of weight $k$ holomorphic cusp forms, explicitly given by
\begin{align*}
    k=12\ell+k', \quad \text{for }k' \in \{0,4,6,8,10,14\}.
\end{align*}
We simply write $f_k:=f_{k,0}$. Duke and Jenkins \cite{DJ08} proved that any zero of $f_{k,m}$ in $F$ lies on the unit circle when $m \geq 0$. This result expands the understanding of the distribution of zeros of modular forms in the vicinity of the unit circle.

On the other hand, there are some results on the zeros lying on another geodesic segment in~$F$. Ghosh and Sarnak \cite{GS12} established estimates for the number of zeros of cuspidal Hecke eigenforms on the union of segments $\delta_1\cup\delta_2\cup\delta_3$, where
\begin{align*}
    \delta_1:=\{z \in F: \Re(z)=0\}, \quad \delta_2:=\{z \in F: \Re(z)=1/2\}, \quad \delta_3:= \{z \in F: |  z|  =1\}.
\end{align*}
They proved that the number of zeros of weight $k$ cuspidal Hecke eigenform on the segment $\delta_2$ is~$\gg \log k$, and on the union of the segments $\delta_1\cup\delta_2$, it is $\gg_{\epsilon} k^{1/4-1/80-\epsilon}$ as $k$ goes to infinity. Matom\"{a}ki \cite{Mat16} later improved by showing that each number is $\gg_{\epsilon} k^{1/8-\epsilon}$, and $\gg_{\epsilon} k^{1/4-\epsilon}$, respectively.

In accordance with Ghosh and Sarnak's terminology, we use the term {\it real zeros} to describe the zeros lying on the aforementioned geodesic segments.

In a recent paper \cite{GO22}, Gun and Oesterl\'{e} proved that the behavior of the derivative of the Eisenstein series $E_k$ for each even integer $k\geq 4$  with respect to the real zeros is interesting. They showed that all the zeros of $E_k'$ in $F$ lie on the line segment $\delta_2$. This finding prompts two natural questions, given the fact that Getz's result in \cite{Get04} and Duke and Jenkins' result in \cite{DJ08} focus on the properties of the zeros of the analogue of $E_k$:

\begin{itemize}
    \item Q1. Is every zero of depth $1$ quasimodular form in $F$ `near' $E_k'$ lying on $\delta_2$?
    \item Q2. Is every zero of $f_k'$  in $F$ lying on $\delta_2$?
\end{itemize}

In this paper, we investigate the properties of the real zeros of depth $1$ quasimodular forms, and address two natural questions related to the behavior of derivatives of modular forms. Our first main result provides an affirmative answer to Q1 when the field of coefficients of forms is restricted to real numbers.

\begin{theorem}\label{thm_near Eisen'} For a positive even integer $k$, let $E_k',g_1,g_2,\ldots,g_n$ be a basis of the space of weight $k+2$ depth $\leq 1$ quasimodular forms with real Fourier coefficients, and let $f=E_k'+\sum_{j=1}^n a_j g_j$ be a depth $\leq 1$ quasimodular form with $a_j \in \R$. If $|a_j|$'s are small enough depending on basis elements $g_1,\ldots,g_n$, then every zero of~$f$ in $F$ lies on $\delta_2$.
\end{theorem}

On the other hand, the answer to Q2 is negative. If we consider a weight $98$ depth $1$ quasimodular form $f_{96}'$, then numerical computations found that it has $4$ real zeros on $\delta_2$ and $2$ non-real zeros in $F$, whose real parts are approximately $0.44$ and $-0.44$, respectively. Nonetheless, we provide partial results on the proportion of real zeros for $f_k'$ under certain conditions regarding the residue of weight.

\begin{theorem}\label{thm_gap ftn}
    Let $k \equiv 0 \pmod{12}$ and let $\theta_j:=\frac{j\pi}{k+1}$,  $t_j:=\frac{1}{2}\cot \theta_j$.
    \begin{enumerate}[\normalfont(a)]
        \item For sufficiently large $k$, the function $f_k'\left(\frac{1}{2}+it\right)$ in variable $t$ has $\gg k$ sign changes along the interval $(\sqrt{3}/2,\infty)$. More precisely, if $k \geq 1116$, then the sign of $f_k'\left(\frac{1}{2}+it_j\right)$ is $(-1)^j$ for~$19(k+1)/50\pi \leq j \leq [k/6]-1$.
        \item For large $k$, over $27.4\%$ zeros of $f_k'$ in $F$ lie on $\delta_2$.
    \end{enumerate}
\end{theorem}

Theorem \ref{thm_gap ftn}(a) does not hold for other values of $k'=4,6,8,10,14$. For instance, if $k=12\ell+14$ is large, our proof of Theorem \ref{thm_gap ftn} shows that the function $Df_{k,0}(z)$ never vanishes for $0.38 \leq \theta \leq \theta_0$ for some appropriate $\theta_0<\pi/6$. We note that $\theta_j$ for $j=19(k+1)/50\pi$ tends to $0.38$ as $k$ approaches~$\infty$. However, this observation does not imply that Theorem \ref{thm_gap ftn}(b) does not hold for other $k'$. In fact, for $k'=14$, we speculate that if $j\ll 19(k+1)/50\pi$, it plays a similar role as in Theorem~\ref{thm_gap ftn}(a).
See Remark~\ref{rmk_for other k'} for further discussion.

We also consider real zeros in a general setting. Specifically, the following result indicates that every quasimodular form  given as the derivative of a cusp form, has real zeros.

\begin{theorem}\label{thm_cusp form 2 crit pt}
    Let $S_{k-2,\R}$ be the space of cusp forms of weight $k-2$ with real Fourier coefficients. Let $f \in DS_{k-2,\R}$ be a quasimodular form of weight $k$ and depth $1$. Then $f$ has at least two real zeros, one lying on the central line $\{z \in \Ha:\Re(z)=0\}$ and the other on $\{z \in \Ha:\Re(z)=1/2\}$.
\end{theorem}

It is noteworthy that Theorem~\ref{thm_cusp form 2 crit pt} does not apply to general quasimodular forms of depth~$1$ (Proposition~\ref{thm_monotonedecreasing noncuspidal}).

Let us briefly outline the contents of this paper. In Section~\ref{sec_valences} and Section~\ref{sec_Serre}, we provide the necessary preliminaries, and Section \ref{sec_thm 1} presents the proof of Theorem \ref{thm_near Eisen'} which affirms the answer to the first question. In Section \ref{sec_thm 2}, we prove Theorem \ref{thm_gap ftn}, which gives a partial result towards answering the second question. Finally, Section \ref{sec_other remarks} offers additional discussions on the real zeros of quasimodular forms on the line $\{z \in \Ha: \Re(z)=1/2\}$ including the proof of Theorem~\ref{thm_cusp form 2 crit pt}.

\section{Valence formula}\label{sec_valences}


To prove our results, we first introduce the {\it valence formula} for depth~$1$ quasimodular forms which is developed in \cite{IR22} recently. The classical valence formula is the equation on the multiplicities of zeros of a weight $k$ modular form $f$ as follows:
\begin{align*}
    \sum_{z \in X} \frac{v_z(f)}{e_z}=\frac{k}{12}.  
\end{align*}
Here, $X$ denotes the (compactified) modular curve $X:=\Gamma \backslash (\Ha \cup \{\infty\})$, $e_z$ is the ramification index of $z$ (which equals~$1$ except for $z=i$, $\rho=e^{i\pi/3}$ and $e_i=2$, $e_{\rho}=3$), and $v_z(f)$ denotes the multiplicity of $f$ at $z$.

To state the valence formula, we introduce some notations. Let $f=f_0+f_1E_2$ be a quasimodular form of weight $k$ and depth~$1$ with real Fourier coefficients for which $f_0$ and $f_1$ are modular forms with real Fourier coefficients and having no common zeros. Let $\varphi_1, \ldots, \varphi_n$ be real numbers such that $\pi/3 \leq \varphi_1 \leq \ldots \leq \varphi_n \leq \pi/2$ and $e^{i\varphi_j}$ are zeros of $f_1$ for all $1 \leq j \leq n$, counted with multiplicity. Denote by $r(f_1)$ the sign of the first nonzero coefficient of the Taylor expansion of $f_1$ around $\rho$ (see \cite[(11)]{IR22} for precise definition), and define $w(z)$ by $1$ except for $z=i,\rho,\infty$, and $2$ for those exceptional points. We define 
\begin{align}\label{eqn_N1(f)}
    N_{\infty}(f)=\frac{1}{2}\left[\frac{k}{6}\right]-(-1)^{v_{\rho}(f_1)}r(f_1)\sum_{j=1}^n \frac{(-1)^j}{w(e^{i\varphi_j})}\sgn \left(e^{\frac{1}{2}ik\varphi_j}f(e^{i\varphi_j})\right).
\end{align}

 We adopt a slightly modified notation for the fundamental domain as follows:
\begin{align*}
    \mathcal F=\{z \in \Ha: 0\leq \Re(z) < 1, \quad |z| \geq 1 \text{ if } 0\leq\Re(z)\leq1/2, \quad |z-1| > 1 \text{ if } 1/2<\Re(z)<1\} \cup \{\infty\},  
\end{align*}
which aids in visualizing the arguments in Section~\ref{sec_thm 1}.

\begin{theorem}\cite[Theorem~1.3]{IR22}\label{thm_valence formula}
    Let $f=f_0+f_1E_2$ be a quasimodular form of weight $k$ and depth~$1$ for which $f_0$ and $f_1$ are modular forms with real Fourier coefficients and having no common zeros. Then there exist constants $N_{\infty}(f)$ that depend on $f$ such that the following holds:
    \begin{align*}
        \sum_{z \in \mathcal F} \frac{v_z(f)}{e_z}= N_{\infty}(f)
    \end{align*}
\end{theorem}

\begin{remark}\label{rmk_F and mathcal F}
    Theorem~\ref{thm_valence formula} is proven more generally in \cite{IR22}; it counts the number of zeros in arbitrary fundamental domains rather than the standard fundamental domain only. We also remark that the definition of the standard fundamental domain $F$ and its modification $\mathcal F$ are slightly different from the standard fundamental domain defined in \cite{IR22}. Nonetheless, Theorem~\ref{thm_valence formula} remains valid since the order of zeros of quasimodular forms remains unchanged under translation $z \mapsto z+1$. Similarly, although the angles $\varphi_1,\ldots,\varphi_n$ in \cite{IR22} are taken to satisfy $\pi/2 \leq \varphi_n \leq \ldots \leq \varphi_1 \leq 2\pi/3$, which differs from our choice, Theorem~\ref{thm_valence formula} still holds since $f_1$ is assumed to have real Fourier coefficients, making the order of its zeros invariant under translation $z \mapsto -\bar{z}$.
\end{remark}


The following is the special case of Theorem \ref{thm_valence formula}.

\begin{theorem}\cite[Theorem~1.1]{IR22}\label{thm_valence formula specialized}
    If $f$ is a modular form of weight $k$ with real Fourir coefficients, then
    \begin{align*}
        \sum_{z \in \mathcal F} \frac{v_z(f')}{e_z}=\frac{k}{12}+
            C(f)+\frac{1}{3}\delta_{f(\rho)=0},
    \end{align*}
    where $C(f)$ is the number of distinct zeros $z$ on the unit circle in $\mathcal F$ counted with weight $e_z^{-1}$, and $\delta_{f(\rho)=0}$ is $1$ if $f(\rho)=0$ and $0$ otherwise.
\end{theorem}

In Theorem~\ref{thm_valence formula specialized}, $\mathcal F$ can be replaced to $F$ as explained in Remark~\ref{rmk_F and mathcal F}. Theorem \ref{thm_valence formula specialized} provides the number of zeros of $f_{k,0}'$ in $F$ as shown in the following proposition:

\begin{proposition}\label{prop_number of zeros of f_k,0}
The number of zeros $z$ of $f_{k}'$ in $F$ counted with multiplicities and with weight~$e_z^{-1}$ is equal to $\left[\frac{k-4}{6}\right]+\delta_{k \equiv 2 \pmod 6}$, where $\delta_{k \equiv 2 \pmod 6}$ is $1$ if $k \equiv 2 \pmod 6$ and $0$ otherwise.

\end{proposition}

\begin{proof}
According to Theorem \ref{thm_valence formula specialized}, the number of zeros of $f_k'$ (counted with multiplicities and with weith $e_z^{-1}$) in $F$ is
\begin{align*}
\frac{k}{12}+C(f_k)+\frac{1}{3}\delta_{f_k(\rho)=0}.
\end{align*}
By \cite[Theorem~1]{DJ08} and the classical valence formula, it is the same as
\begin{align*}
\frac{k}{12}+C(E_k)+\frac{1}{3}\delta_{E_k(\rho)=0},
\end{align*}
which is the number of zeros of $E_k'$ (counted with multiplicities and with weight $e_z^{-1}$) in $F$. Referring to \cite[Theorem~2]{GO22}, it equals $\left[\frac{k-4}{6}\right]+\delta_{k \equiv 2 \pmod 6}$. 
\end{proof}

\section{Preliminaries}\label{sec_Serre}

Let $D:=\frac{1}{2\pi i}\frac{d}{dz}=q\frac{d}{dq}$ be the normalized derivation. The space of quasimodular forms of weight $k$ can be written as (cf. \cite[Theorem~4.2]{Roy12})
\begin{align*}
    \widetilde{M}_k=\widetilde{M}_k^{(\leq k/2)}=\bigoplus_{j=0}^{\frac{k}{2}-2} D^j M_{k-2j} \oplus \C D^{\frac{k}{2}-1}E_2.
\end{align*}
We further have
\begin{align*}    \widetilde{M}_{k,\R}=\widetilde{M}_{k,\R}^{(\leq k/2)}=\bigoplus_{j=0}^{\frac{k}{2}-2} D^j M_{k-2j,\R} \oplus \R D^{\frac{k}{2}-1}E_2.
\end{align*}
Indeed, if we consider an involution $\imath : f(\tau) \mapsto \overline{f(-\overline{\tau})}$ on $\widetilde{M}_{k}^{(\leq k/2)}$ which is an anti-linear map, then each direct summand $D^j M_{k-2j}$ and $\C D^{\frac{k}{2}-1}E_2$ in the above are invariant under $\imath$. On the other hand, the subspace $\widetilde{M}_{k}^{\imath}$ of $\widetilde{M}_{k}$ consisting of the elements fixed by $\imath$ is exactly~$\widetilde{M}_{k,\R}$. Since $D=q\frac{q}{dq}$, if we write $f=\sum_{j=0}^{\frac{k}{2}-2}D^j f_j+cD^{\frac{k}{2}-1}E_2 \in \widetilde{M}_{k,\R}$ where $f_j \in M_{k-2j}$ and $c \in \mathbb C$, then 
\begin{align*}
    f=\imath(f)=\sum_{j=0}^{\frac{k}{2}-2}D^j \imath(f_j)+\overline{c}D^{\frac{k}{2}-1}E_2,
\end{align*}
so $f_j=\imath(f_j)$ and $\overline{c}=c$, i.e., $f_j \in M_{k-2j,\mathbb R}$ and $c \in \mathbb R$.


\

We define the Serre derivative for arbitrary depth $p$ quasimodular forms by $\vartheta:=D-\frac{k-p}{12}E_2$ (the notation $\vartheta$ is different from $\theta$ which usually denotes the angle in this paper. Some authors used $\theta$ to denote the normalized derivation, which is not adapted in this paper). 

\begin{lemma}\label{lem_thetafhasniceform-arbitrarydepth}
Let $f$ be a quasimodular form of weight $k$ and depth $p$. Then $\vartheta f$ are quasimodular form of weight $k+2$ and depth $p$. Furthermore, $f$ is cuspidal if and only if $\vartheta f$ is cuspidal.
\end{lemma}

\begin{proof}
The assertion is well-known in the literature, but we record the proof for clarity. For $p=0$, see \cite[p.48]{BVHZ08}. For $p \geq 1$, write
$f=f_0 E_2^{p}+f_1$
for a modular form $f_0$ of weight $k-2p$ and a quasimodular form $f_1$ of weight $k$ and depth $\leq p-1.$ Then
\begin{align*}
    D f&=D (f_0 E_2^{p}) + D f_1 \\
    &=(D f_0) E_2^{p}+\frac{p f_0}{12}E_2^{p-1}(E_2^2-E_4)+D f_1 \\
    &=\frac{k-p}{12}E_2(f_0E_2^{p}+f_1)+\left((\vartheta f_0)E_2^{p} + D f_1 -\frac{k-p}{12}E_2f_1-\frac{p}{12}E_2^{p-1}E_4f_0. \right) \\
    &=\frac{k-p}{12}E_2f+g.
\end{align*}
One can verify that $\vartheta f=g:=(\vartheta f_0)E_2^{p} + D f_1 -\frac{k-p}{12}E_2f_1-\frac{p}{12}E_2^{p-1}E_4f_0$ is a quasi-modular form of depth $p$.

Recall that $D f$ is always cuspidal by the definition of $D$. Hence
\begin{align*}
    \lim_{z \to i\infty} \vartheta f(z)&=\lim_{z \to i\infty}\left(D f(z)-\frac{k-p}{12}E_2(z)f(z)\right) 
    =-\frac{k-p}{12}\lim_{z \to i\infty}f(z),
\end{align*}
so $f$ is cuspidal if and only if $\vartheta f$ is cuspidal.
\end{proof}

Thus, if we write $DE_k=f_0+f_1 E_2$ for modular forms $f_0$ and $f_1$, then $f_0=\vartheta E_k$ and $f_1=\frac{k}{12}E_k$.

\section{proof of Theorem \ref{thm_near Eisen'}}\label{sec_thm 1}

\ww{

In this section, we present the proof of Theorem~\ref{thm_near Eisen'}. Consider the space $\widetilde M_{k}^{(\leq p)}$ quasimodular forms of weight $k$ and depth~$\leq p$. Let $\widetilde M_{k,\R}^{(\leq p)}$ be the set of quasimodular forms of weight $k$ and depth~$\leq p$ with real Fourier coefficients. Then there is a natural topology on $\widetilde M_{k,\R}^{(\leq p)}$ (resp. $\widetilde M_{k}^{(\leq p)}$) which makes $\widetilde M_{k,\R}^{(\leq p)}$ (resp. $\widetilde M_{k}^{(\leq p)}$) the canonical topological $\mathbb R$ (resp. $\mathbb C$)-vector space.

Now, let $C$ be a subset of $\widetilde M_{k+2,\R}^{(\leq 1)}$ given by}
\begin{align*}
C:=\{f \in \widetilde M_{k+2,\R}^{(\leq 1)} : \text{if $f(z)=0$ and $z \in \mathcal F$, then $\Re(z)=1/2$}\}.
\end{align*}
We equip $C$ with the subspace topology inherited from $\widetilde M_{k+2,\R}^{(\leq 1)}$. Since $DE_k \in C$, it suffices to show the following to prove Theorem \ref{thm_near Eisen'}.

\begin{theorem}\label{prop_Ek in int CM}
If $k \geq 4$ is an even integer, then $D E_k$ belongs to the interior $\inte(C)$ of $C$.
\end{theorem}

We recall the following theorem due to Getz with suitably modified but equivalent arc boundary.

\begin{theorem}{\cite[Theorem~1]{Get04}}\label{thm_Getz}
    Let $k\geq 4$ be an even  integer. Then, there is an open neighborhood $U \subseteq M_{k,\mathbb R}$ of $E_k$ such that if $g \in U$ has $\ell=\ell(k)$ simple zeros on $\{e^{i\theta}: \pi/3<\theta < \pi/2\}$ and the following trivial zeros of $g$ depending on $k$ modulo $12$:
    \begin{align*}
        &v_i(g)=\begin{cases}
            1 & k \equiv 2 \pmod 4, \\
            0 & k \equiv 0 \pmod 4,
        \end{cases} \\
        &v_{\rho}(g)=\begin{cases}
            2 & k \equiv 2 \pmod 6, \\
            1 & k \equiv 4 \pmod 6, \\
            0 & k \equiv 0 \pmod 6.
        \end{cases}
    \end{align*}
\end{theorem}

Moreover, these zeros behave analytically when we move the modular forms $g$ smoothly in the following sense.

\begin{lemma}\label{lem_E_k nbhd zero simple}
    Let $k\geq 4$ be an even integer. Let $U \subseteq M_{k,\mathbb R}$ be an open neighborhood defined in Theorem~\ref{thm_Getz} and $\varphi_1(g),\ldots,\varphi_{\ell}(g)$ denotes the angle associated with the $\ell$ simple zeros of $g\in U$ on $\{e^{i\theta}: \pi/3 <\theta < \pi/2\}$. Equip the smooth structure on $M_{k,\mathbb R}$ using the isomorphism $M_{k,\mathbb R} \simeq \mathbb R^{\ell+1}$ of $\mathbb R$-vector spaces. Then $\varphi_i(g)$ for $1\leq i \leq \ell$ are locally smooth functions.
\end{lemma}

\begin{proof}
   Let $g_1, \ldots, g_d \in M_{k}$ form a basis for $M_{k}$. Define a linear map $g : \C^d \to M_{k}$ by $g(c_1,\ldots,c_d)=c_1g_1+\cdots+c_d g_d$ that is an isomorphism of topological $\C$-vector spaces, and also define $\tilde{g}:\Ha \times \C^d \to \C$ by $\tilde g(z,c_1,\ldots,c_d)=g(c_1,\ldots,c_d)(z)$. Let $g(c_{0,1},\ldots,c_{0,d})=E_k$.

Consider a zero $(z_0,c_{0,1},\ldots,c_{0,d}) \in \Ha \times \C^d$ of $\tilde g$. Since all the zeros of $E_k$ are simple, $z_0$ is a simple zero of $g(c_{0,1},\ldots,c_{0,d})$. Thus we have 
\begin{align*}
    \det\left(\frac{\partial \tilde g}{\partial z}(z_0,c_{0,1},\ldots,c_{0,d})\right)=\frac{\partial \tilde g}{\partial z}(z_0,c_{0,1},\ldots,c_{0,d})\neq 0,
\end{align*}
so by the analytic implicit function theorem, there exist open neighborhoods $U$ of $(c_{0,1},\ldots,c_{0,d}) \in \C^d$ and $V$ of $z_0 \in \Ha$, and an analytic function $\xi : U \to V$ such that
\begin{align*}
\tilde g(\xi(c_1,\ldots,c_d), c_1,\ldots,c_d)&=0, \quad \text{if } (c_1,\ldots,c_d) \in U, \\
\xi(c_{0,1},\ldots,c_{0,d})&=z_0.
\end{align*}
This means that $\xi$ provides the local zero locus of $g$ near $z_0$. In particular, for simple zeros $z_1=e^{i\varphi_1},\ldots,z_{\ell}=e^{i\varphi_{\ell}}$ for $\pi/3<\varphi_1,\ldots,\varphi_{\ell}<\pi/2$, we have analytic loci $\xi_j : U \to V_j \subset \mathbb H$ with $\xi_j(c_{0,1},\ldots,c_{0,d})=z_j$, $1\leq j \leq \ell$. 

Identifying $\R^d$ with a subset of $\left(\R+i\cdot 0\right)^d$ in $\C^d$, note that $\xi|_{U \cap \R^d}$ is a smooth function. According to Theorem~\ref{thm_Getz}, we may assume by shrinking $U$ that these $\xi_i(c_1,\ldots,c_d)$ lie on $\{e^{i\theta}: \pi/3<\theta<\pi/2\}$ and provide all the zeros of $g(c_1,\ldots,c_d)$ except for trivial zeros. Therefore, we have $e^{i\varphi_j(g(c_1,\ldots,c_d))}=\xi_j(c_1,\ldots,c_d)$ (after change the order of indices) and hence $\varphi_j(g)$ are smooth for~$g\in U \cap \mathbb R^d$.
\end{proof}

\begin{proof}[Proof of Theorem~\ref{prop_Ek in int CM}]
Let $f_1, \ldots, f_d \in \widetilde M_{k+2,\R}^{(\leq 1)}$ form a basis for $\widetilde M_{k+2}^{(\leq 1)}$. Define a linear map $f : \C^d \to \widetilde M_{k+2}^{(\leq 1)}$ by $f(c_1,\ldots,c_d)=c_1f_1+\cdots+c_d f_d$ that is an isomorphism of topological $\C$-vector spaces, and also define $\tilde{f}:\Ha \times \C^d \to \C$ by $\tilde f(z,c_1,\ldots,c_d)=f(c_1,\ldots,c_d)(z)$.

Consider a zero $(z_0,c_{0,1},\ldots,c_{0,d}) \in \Ha \times \C^d$ of $\tilde f$ and suppose that $z_0$ is a simple zero of $f(c_{0,1},\ldots,c_{0,d})$. Then, 
\begin{align*}
    \det\left(\frac{\partial \tilde f}{\partial z}(z_0,c_{0,1},\ldots,c_{0,d})\right)=\frac{\partial \tilde f}{\partial z}(z_0,c_{0,1},\ldots,c_{0,d})\neq 0,
\end{align*}
so by the analytic implicit function theorem, there exist open neighborhoods $U$ of $(c_{0,1},\ldots,c_{0,d}) \in \C^d$ and $V$ of $z_0 \in \Ha$, \ww{and an analytic function $\xi : U \to V$ such that
\begin{align*}
\tilde f\left(\xi(c_1,\ldots,c_d), c_1,\ldots,c_d\right)&=0 \text{ if } (c_1,\ldots,c_d) \in U \quad \text{ and } \quad 
\xi(c_{0,1},\ldots,c_{0,d})=z_0.
\end{align*}}
This means that $\xi$ provides the local zero locus of $f$ near $z_0$. Identifying $\R^d$ with a subset of $\left(\R+i\cdot 0\right)^d$ in $\C^d$, the function $\xi|_{U \cap \mathbb R^d}$ is smooth.

Take $c_{0,1},\ldots,c_{0,d} \in \R$ so that $f(c_{0,1},\ldots,c_{0,d})=D E_k$. First, suppose $k \not\equiv 2 \pmod{6}$. There are $m:=\left[\frac{k-4}{6}\right]$ distinct zeros $z_1,\ldots,z_m$ in $\mathcal F$ which all lie on $\{z \in \mathcal F : \Re(z)=1/2\}$. These zeros are all simple (\cite[Theorem 4]{GO22}). As we have seen above, one can take open sets $U_j \subseteq \C^d$, $V_j \subseteq \Ha$, and analytic functions $\xi_j: U_j \to V_j$ for each $z_j$. By shrinking $U_j$ and $V_j$, we may assume the following conditions: let $U:=\cap_{j=1}^m U_j$.
\begin{enumerate}
    \item[\normalfont(i)] $V_i$'s are pairwise disjoint,
    \item[\normalfont(ii)] for any $(c_1,\ldots,c_d) \in U$, the zeros of $f(c_1,\ldots,c_d)$ have the absolute value larger than $1$,    
    \item[\normalfont(iii)] for any $(c_1,\ldots,c_d) \in U \cap \R^d$, if we write $$f(c_1,\ldots,c_d)=f_0(c_1,\ldots,c_d)+f_1(c_1,\ldots,c_d)E_2 \in M_{k+2,\mathbb R} \oplus M_{k,\mathbb R}E_2,$$ then $f_1(c_1,\ldots,c_d)(z)$ has the same number of zeros (counted with multiplicity) on $\{z \in \mathcal F: |z|=1 \text{ and } 0 \leq \Re(z) \leq 1/2\}$, say $n$, which is the number of zeros of $f_1(c_{0,1},\ldots,c_{0,d})=~\frac{k}{12}E_k$, and 
    \begin{align*}
        w(e^{i\varphi_j(c_1,\ldots,c_d)})&=w(e^{i\varphi_j(c_{0,1},\ldots,c_{0,d})}), \\
        v_{\rho}(f_1(c_1,\ldots,c_d))&=v_{\rho}\left(\frac{k}{12}E_k\right), \\
        r(f_1(c_1,\ldots,c_d))&=r\left(\frac{k}{12}E_k\right), \\ \sgn\left(e^{\frac{1}{2}i(k+2)\varphi_j(c_1,\ldots,c_d)}f(c_1,\ldots,c_d)(e^{i\varphi_j(c_1,\ldots,c_d)})\right)&=\sgn\left(e^{\frac{1}{2}i(k+2)\varphi_j(c_{0,1},\ldots,c_{0,d})}\frac{k}{12}E_k(e^{i\varphi_j(c_{0,1},\ldots,c_{0,d})})\right),
    \end{align*}
     where $\varphi_j,w,v_{\rho},r$ are defined in Section~\ref{sec_valences}.
\end{enumerate}
Let us elaborate on why condition (iii) can be satisfied. The consistency of the number of zeros (counted with multiplicity) of $f_1(c_1,\ldots,c_d)$ follows from Theorem~\ref{thm_Getz}; we may shrink $U$ so that $f(U)\cap M_{k,\mathbb R}E_2$ is an open subset of $M_{k,\mathbb R}E_2 \simeq M_{k,\mathbb R}$ satisfying Theorem~\ref{thm_Getz}. The first and second equations
\begin{align*}
    w(e^{i\varphi_j(c_1,\ldots,c_d)})&=w(e^{i\varphi_j(c_{0,1},\ldots,c_{0,d})}), \\
    v_{\rho}(f_1(c_1,\ldots,c_d))&=v_{\rho}\left(\frac{k}{12}E_k\right)
\end{align*}
are met according to Theorem~\ref{thm_Getz} again. More precisely, when $(c_1,\ldots,c_d)$ varies within sufficiently small open set, the values of $v_{\rho}(f_1(c_1,\ldots,c_d))$ and $v_{i}(f_1(c_1,\ldots,c_d))$ are determined by the weight of $f_1(c_1,\ldots,c_d)$. Since $w(z)= 2$ for $z\in \mathcal F$ if and only if $z=i,\rho,\infty$, we have desired equations. The third equation
\begin{align*}
    r(f_1(c_1,\ldots,c_d))=r\left(\frac{k}{12}E_k\right)    
\end{align*}
is a consequence of the continuity of the first nonzero coefficient of the Taylor expansion of $f_1(c_1,\ldots,c_d)$ around $\rho$. Since $r(f_1(c_1,\ldots,c_d))$ represents the sign of this nonzero coefficient, it remains constant as $(c_1,\ldots,c_d)$ varies within a sufficiently small open set. The last equation stems from the fact that $\varphi_j(c_1,\ldots,c_d)$ are locally smooth according to Lemma~\ref{lem_E_k nbhd zero simple}. To avoid confusion, we note that we can assume 
\begin{align*}
    e^{\frac{1}{2}i(k+2)\varphi_j(c_1,\ldots,c_d)}f(c_1,\ldots,c_d)(e^{i\varphi_j(c_1,\ldots,c_d)})
\end{align*}
is nonzero; if it were to vanish, then $e^{i\varphi_j(c_1,\ldots,c_d)}$ would be a common zero of $f_0$ and $f_1$. As $f(c_{0,1},\ldots,c_{0,d})=DE_k$, we can avoid this situation by the local continuity of the zeros of $f_0(c_1,\ldots,c_d)$ and $f_1(c_1,\ldots,c_d)$ as per Lemma~\ref{lem_E_k nbhd zero simple} together with the fact that $f_0(c_{0,1},\ldots,c_{0,d})$ and $f_1(c_{0,1},\ldots,c_{0,d})$ do not have the common zero. Therefore, the sign of $e^{\frac{1}{2}i(k+2)\varphi_j(c_1,\ldots,c_d)}f(c_1,\ldots,c_d)(e^{i\varphi_j(c_1,\ldots,c_d)})$ remains constant as $(c_1,\ldots,c_d)$ varies within a sufficiently small open set. 

Since the set
\begin{align*}
\left\{(c_1,\ldots,c_d) \in \C^d : \frac{\partial \tilde f}{\partial z}\left(\xi_i(c_1,\ldots,c_d),c_1,\ldots,c_d\right)=0\right\}
\end{align*}
is a closed measure-zero set, we may further assume that $\xi_i(c_1,\ldots,c_d)$ is a simple zero of $f(c_1,\ldots,c_d)$ for all $(c_1,\ldots,c_d) \in U$. We assumed each ingredient of the right-hand side of \eqref{eqn_N1(f)} is constant for $f=f(c_1,\ldots,c_d)$ when $(c_1,\ldots,c_d) \in U \cap \mathbb R^d$, so by Theorem~\ref{thm_valence formula} and the fact that $e_z>1$ if $|z|>1$, we obtain the following property:
\begin{enumerate}[(P1)]
    \item If $(c_1,\ldots,c_d) \in U \cap \R^d$, then the number of zeros (counted with multiplicity) of $f(c_1,\ldots,c_d)$ in $\mathcal F$ is $m$.  
\end{enumerate}

Let $U_0$ be the connected component of $U \cap \R^d$ containing $D E_k$, which is also open in $\R^d$ due to the local connectedness of $\R^d$. Clearly, (P1) holds if $(c_1,\ldots,c_d) \in U_0$. Meanwhile, as we have taken above, there are $m$ simple zero loci $\xi_1(c_1,\ldots,c_d),\ldots,\xi_m(c_1,\ldots,c_d)$ for $(c_1,\ldots,c_d) \in U$. Combining this fact with (P1), we have the following conclusion:
\begin{enumerate}[(P2)]
    \item If $(c_1,\ldots,c_d) \in U_0$, then $\xi_1(c_1,\ldots,c_d),\ldots,\xi_m(c_1,\ldots,c_d)$ provide all $m$ simple zeros of $f(c_1,\ldots,c_d)$ in $\mathcal F$.
\end{enumerate}

We claim that $U_0 \subseteq C$ so that $D E_k \in U_0 \subseteq \inte(C)$, which completes the proof.

\ww{Suppose for contradiction that there exists a point $(c_{1,1},\ldots,c_{1,d}) \in U_0$ such that for some $z_0 \in \mathcal F \setminus \{z \in \mathcal F : \Re(z)=1/2\}$ one has $\tilde f(z_0,c_{1,1},\ldots,c_{1,d})=0$. By (P2), $z_0=\xi_a(c_{1,1},\ldots,c_{1,d})$ for some $1\leq a \leq m$. Note that since $c_{1,1},\ldots,c_{1,d} \in \mathbb R$, $1-\overline{z_0}$ is also a zero of $f(c_{1,1},\ldots,c_{1,d})$ so that $1-\overline{z_0}=\xi_b(c_{1,1},\ldots,c_{1,d})$ for some $b \in \{1,\ldots,m\}$. Since $\Re(z_0) \neq 1/2$, we have $a \neq b$.

Let $p : [0,1] \to U_0$ be a path with $p(0)=(c_{0,1},\ldots,c_{0,d})$ and $p(1)=(c_{1,1},\ldots,c_{1,d})$. For any $t\in[0,1]$, $\xi_a(p(t))$ is a zero of $f(p(t))$, implying that $1-\overline{\xi_a}(p(t))$ is also a zero of $f(p(t))$. Thus, $1-\overline{\xi_a}(p(t))=\xi_{b'}(p(t))$ for some $1\leq b' \leq m$. Define $I_{b'}$ as the set consisting of $t\in [0,1]$ for which $1-\overline{\xi_a}(p(t))=\xi_{b'}(p(t))$, and let $b_1,\ldots,b_r \in \{1,2,\ldots,m\}$ be the set of integers such that $I_{b_i}\neq \emptyset$. We then have
\begin{align*}
    S:=1-\overline{\xi_a}(p([0,1]))=\bigcup_{i=1}^{r} \xi_{b_i}(p(I_{b_i})).
\end{align*}
Consequently,
\begin{align*}
    S \subset V_{b_1} \sqcup \ldots \sqcup V_{b_r},
\end{align*}
or equivalently,
\begin{align*}
    S=(V_{b_1} \cap S) \sqcup \ldots \sqcup (V_{b_r} \cap S).
\end{align*}
Given that $\xi_a$ and $p$ are continuous functions and $[0,1]$ is connected, the set $S$ is also connected, so $r=1$. Since $\xi_b(p(1))=1-\overline{\xi_a}(p(1))$, we have $I_b \neq \emptyset$, leading to $b_1=b$. 

Therefore, we have $1-\overline{\xi_a}(p(t))=\xi_b(p(t))$ for any $t\in[0,1]$, particularly $1-\overline{\xi_a}(p(0))=\xi_b(p(0))$. Recall that $\Re(\xi_a(p(0)))=1/2$ as $p(0)=(c_{0,1},\ldots,c_{0,d}) \in C$. Hence, we obtain
\begin{align*}
    \xi_a(p(0))=1-\overline{\xi_a}(p(0))=\xi_b(p(0)),
\end{align*}
leading to a contradiction.}

Now suppose that $k \equiv 2 \pmod{6}$. In this case, there is an additional simple zero $z_{m+1}$ of $D E_k$ on $\delta_2$, namely $z_{m+1}=\rho$, which is a double zero of $E_k$. We can use the argument as before to show that there is an analytic function $\xi_{m+1} : U_{m+1} \to V_{m+1}$ with $\xi_{m+1}(c_{0,1},\ldots,c_{0,d})=\rho$. Note that a quasimodular form of weight $k+2$ and depth $1$ always vanishes at $\rho$ by the classical valence formula. By shrinking $V_{m+1}$, the only zero of $f(c_1,\ldots,c_d)$ in $V_{m+1}$ is $\xi_{m+1}(c_1,\ldots,c_d)$, which implies that $\xi_{m+1}$ is a constant function. Therefore, we can use the same argument for the other zeros (which have the absolute value larger than $1$) as before to show that $U_0 \subseteq C$.
\end{proof}

We remark that the same argument as above proves the following.

\begin{proposition}
Let $f=f_0+f_1 E_2$ be a quasimodular form of weight $k$, where $f_0$ and $f_1$ are modular forms such that $f_1(z) \neq 0$ for all $z \in F$ with $|  z|  =1$. If $f$ is on the topological boundary of~$C$, then there is a multiple zero of $f$ in $F$.
\end{proposition}

\begin{proof}
Note that either $f\in C$ or $f \in \widetilde M_{k,\R}^{(\leq 1)} \setminus C$. Assume that all the zeros of $f$ in $F$ are simple. By the analytic implicit function theorem and the same argument as in the proof of Theorem~\ref{prop_Ek in int CM}, if $f \in C$ (resp. $f \in \widetilde M_{k,\R}^{(\leq 1)}\setminus C$) then $f \in \mathrm{int}\left(C\right)$ (resp. $f \in \mathrm{cl}\left(\widetilde M_{k,\R}^{(\leq 1)}\setminus C\right)$) which implies that $f \not\in \mathrm{cl}\left(\widetilde M_{k,\R}^{(\leq 1)}\setminus C\right)$ 
 (resp. $f \not\in \mathrm{cl}\left(C\right)$).
\end{proof}

\section{Proof of Theorem \ref{thm_gap ftn}}\label{sec_thm 2}

Our main goal in this section is to prove Theorem~\ref{thm_gap ftn}. We investigate the value of $Df_{k,m}(z)$ at $z=\frac{1}{2}+\frac{i}{2}\cot \theta=ie^{-i\theta}|z|  $ for $\theta \in \left[\frac{(\ell+1)\pi}{k+1},\frac{([k/6]-1)\pi}{k+1}\right]$ and consequently at $z=\frac{1}{2}+it_j$ for $19(k+1)/50\pi \leq j \leq [k/6]-1$. The strategy is based on an integral representation of $f_{k,m}$ which is equivalent to the generating function formula for $f_{k,m}$ established in \cite{DJ08}. 

Let us define the function $H(\tau,z)$ as follows:
\begin{align*}
    H(\tau,z):=&H_1(\tau,z)+H_2(\tau,z),
\end{align*}
where
\begin{align*}
    H_1(\tau,z):=&-\frac{1}{2\pi i}\ell E_2(z) \frac{\Delta^{\ell}(z)}{\Delta^{\ell}(\tau)}\frac{E_{k'}(z)}{E_{k'}(\tau)}\frac{\frac{d}{d\tau}(j(\tau)-j(z))}{j(\tau)-j(z)}e^{-2\pi im\tau} \\
    &-\frac{1}{2\pi i}\frac{\Delta^{\ell}(z)}{\Delta^{\ell}(\tau)}\frac{DE_{k'}(z)}{E_{k'}(\tau)}\frac{\frac{d}{d\tau}(j(\tau)-j(z))}{j(\tau)-j(z)}e^{-2\pi im\tau}
\end{align*}
and
\begin{align*}
    H_2(\tau,z):=\frac{1}{4\pi^2}\frac{\Delta^{\ell}(z)}{\Delta^{\ell}(\tau)}\frac{E_{k'}(z)}{E_{k'}(\tau)}\frac{j'(z)j'(\tau)}{(j(\tau)-j(z))^2}e^{-2\pi im\tau}.
\end{align*}

Note that $\Im(z)=\frac{1}{2}\cot \theta>\frac{\sqrt 3}{2}=\Im(\rho)$.

\begin{lemma}\label{lem_DJ08 lemma2}\label{lem_integration A'}
    Let $A' \in \left[\left.\frac{\sqrt 3}{2},\frac{1}{2}\cot \theta\right)\right.$. For any $m \in \Z$ and sufficiently large $A>0$, we have
    \begin{align}\label{eqn_Df integral expression}
        Df_{k,m}(z)=\int_{-\frac{1}{2}+iA}^{\frac{1}{2}+iA} H(\tau,z)d\tau
    \end{align}
    and
    \begin{align}\label{eqn_integral Htauz}
        \int_{-\frac{1}{2}+iA'}^{\frac{1}{2}+iA'} H(\tau,z)d\tau=Df_{k,m}(z)+g_{k'}(z),
    \end{align}
    where
    \begin{align*}
        g_{k'}(z):=&\frac{e^{-2\pi i m z}}{E_{14}(z)E_{k'}(z)}\left[DE_{14}(z)E_{k'}(z)^3(E_{k'}(z)-1)+DE_{k'}(z)E_{14}(z)(E_{k'}(z)^2-1)\right.\\
        &\left.-E_2(z)E_{14}(z)E_{k'}(z)(E_{k'}(z)^3-(\ell+1)E_{k'}(z)^2+\ell)+mE_{14}(z)E_{k'}(z)^3\right].
    \end{align*}
    In particular,  if $k'=0$ we have
    \begin{align*}
        \int_{-\frac{1}{2}+iA'}^{\frac{1}{2}+iA'} H(\tau,z)d\tau=Df_{k,m}(z)+me^{-2\pi i m z}.
    \end{align*}
\end{lemma}

\begin{proof}
    The integral expression \eqref{eqn_Df integral expression} can be readily obtained by taking the derivative of the equation in \cite[Lemma 2]{DJ08}.

    To prove \eqref{eqn_integral Htauz}, note that for sufficiently large $A$ and $A' \in \left[\left.\frac{\sqrt 3}{2},\frac{1}{2}\cot \theta\right)\right.$, the only possible poles of~$H(\tau,z)$ in the rectangle $\{\tau \in \Ha : |  \Re(\tau)|  \leq 1/2, \text{ } A'\leq \Im(\tau) \leq A\}$ are $\tau=z$ and $\tau=z-1$. Indeed, this assertion can be easily verified as follows; if $\tau$ is a pole of $H(\tau,z)$, then $j(\tau)=j(z)$. Note that for given $c \in \mathbb C$, there is a unique $\tau \in F$ such that $j(\tau)=c$. Thus, $\tau=\gamma z=\gamma\left(\frac{1}{2}+\frac{i}{2}\cot \theta\right)$ for~$\gamma \in \Gamma$. Among these $\tau$, the ones belong to our rectangle are $\tau=z$ and $\tau=z-1$.
    
    By vertically shifting the contour of integration in \eqref{eqn_Df integral expression} to a lower height $A'$, while ensuring the avoidance of the poles $z$ and $z-1$ through the inclusion of a clockwise circular path encircling both points, 
   we obtain
    \begin{align*}
        \int_{-\frac{1}{2}+iA'}^{\frac{1}{2}+iA'} H(\tau,z)d\tau&=Df_{k,m}(z)+\pi i \text{Res}_{\tau=z} H(\tau,z)+\pi i \text{Res}_{\tau=z-1} H(\tau,z) \\
        &=Df_{k,m}(z)+2\pi i \text{Res}_{\tau=z}H(\tau,z).
    \end{align*}
    The proof is completed by invoking the relation $Dj(z)=-E_{14}(z)/\Delta(z)$ and evaluating the residues of $H(\tau,z)$ at $\tau=z$.
\end{proof}

    

Let $z_0:=\frac{1}{2}+\frac{i}{2}\cot\left(\frac{\pi}{12}\right)$. Note that $\theta\geq \frac{\ell+1}{k+1}\pi > \frac{\pi}{12}$ implies $\Im(z)<\Im(z_0)$. Hence, we have $\Im(Sz)>\Im(Sz_0)=\sin \left(\frac{\pi}{6}\right)=\frac{1}{2}>\frac{1}{3}$. On the other hand, $\Im(Sz)=\frac{\Im(z)}{|z|^2}=\frac{2\cot \theta}{1+\cot^2\theta}=\sin 2\theta$, so we have $\sin 2\theta=\Im(Sz)>1/3$.

\begin{lemma}\label{lem_H'' approximation}
    For $A' \in \left[\left.\frac{\sqrt 3}{2},\frac{1}{2}\cot \theta\right)\right.$ and $A'' \in \left(\frac{1}{3}, \sin 2\theta \right)$, we have
    \begin{align*}
        \int_{-\frac{1}{2}+iA'}^{\frac{1}{2}+iA'} H(\tau,z)d\tau=&\int_{-\frac{1}{2}+iA''}^{\frac{1}{2}+iA''} H(\tau,z)d\tau \\
        &-2mi^{-k-2}|  z|  ^{-k-2}e^{2\pi m \sin 2\theta}\cos\left((k+2)\theta +4\pi m \sin^2 \theta\right) \\
        &+\frac{k}{\pi}i^{-k}|  z|  ^{-k-1}e^{2\pi m \sin 2\theta}\cos\left((k+1)\theta +4\pi m \sin^2 \theta\right).
    \end{align*}
    In particular, if $m=0$, then we have
    \begin{align*}
        \int_{-\frac{1}{2}+iA'}^{\frac{1}{2}+iA'} H(\tau,z)d\tau=&\int_{-\frac{1}{2}+iA''}^{\frac{1}{2}+iA''} H(\tau,z)d\tau +(-1)^{k/2}\frac{k}{\pi}|  z|  ^{-k-1}\cos\left((k+1)\theta\right).
    \end{align*}
\end{lemma}

\begin{proof}
    We observe that $$\text{sup}_{y \geq \sqrt 3/2}\Im\left(ST^{-1}\left(-\frac{1}{2}+iy\right)\right)=\text{sup}_{y \geq \sqrt 3/2}\Im\left(S\left(-\frac{3}{2}+iy\right)\right)=\frac{1}{3}.$$
    Thus, for any $\tau$ with $\Re(\tau)=-1/2$ and $\Im(\tau)\geq \frac{\sqrt{3}}{2}$, we have $\Im(ST^{-1}(\tau)) \leq 1/3$. Similarly, we have $\Im(ST(\tau)) \leq 1/3$ when $\Re(\tau)=1/2$ and $\Im(\tau)\geq \frac{\sqrt3}{2}$. Consequently, in the region $\{\tau \in \Ha : |  \Re(\tau)|  \leq 1/2, \text{ } A''\leq \Im(\tau) \leq A'\}$, the only poles of $H(\tau,z)$ are located at $\tau=ST^{-1}z=\frac{-1}{z-1}$ and $\tau=Sz=\frac{-1}{z}$ (see Figure~\ref{fig_poles in the region}).

\begin{figure}[h]
\ctikzfig{zeros_fig1}
\caption{Poles in the region $\{\tau \in \Ha : |  \Re(\tau)|  \leq 1/2, \text{ } A''\leq \Im(\tau) \leq A'\}$}\label{fig_poles in the region}
\end{figure}
    
    By reducing the height of the integration to $A''$, we obtain
    \begin{align*}
        \int_{-\frac{1}{2}+iA'}^{\frac{1}{2}+iA'} H(\tau,z)d\tau=\int_{-\frac{1}{2}+iA''}^{\frac{1}{2}+iA''} H(\tau,z)d\tau-2\pi i \left(\text{Res}_{\tau=\frac{-1}{z-1}}H(\tau,z)+\text{Res}_{\tau=\frac{-1}{z}}H(\tau,z)\right).
    \end{align*}
    Note that for $z=\frac{1}{2}+\frac{i}{2}\cot \theta$, we have 
    \begin{align*}
        z&=ie^{-i\theta}|  z|  , 
        &z-1&=ie^{i\theta}|  z|  , \\
        -(z-1)^{-1}&=2i(\sin \theta) e^{-i\theta},
        &-z^{-1}&=2i(\sin \theta)e^{i\theta}.
    \end{align*}

Since $H_1(\tau,z)$ has a simple pole at $\tau=-1/(z-1)$, we can compute the residue as follows. 
\begin{align*}
    &\text{Res}_{\tau=\frac{-1}{z-1}}H_1(\tau,z)=\lim_{\tau \to -1/(z-1)}\left(\tau+\frac{1}{z-1}\right)H_1(\tau,z) \\
    &=-\frac{1}{2\pi i}\ell E_2(z) \frac{\Delta^{\ell}(z)}{\Delta^{\ell}(-1/(z-1))}\frac{E_{k'}(z)}{E_{k'}(-1/(z-1))}e^{-2\pi im(-1/(z-1))}j'\left(-\frac{1}{z-1}\right)\lim_{\tau \to -1/(z-1)}\frac{\tau+\frac{1}{z-1}}{j(\tau)-j(z)} \\
    &-\frac{1}{2\pi i}\frac{\Delta^{\ell}(z)}{\Delta^{\ell}(-1/(z-1))}\frac{DE_{k'}(z)}{E_{k'}(-1/(z-1))}e^{-2\pi im(-1/(z-1))}j'\left(-\frac{1}{z-1}\right)\lim_{\tau \to -1/(z-1)}\frac{\tau+\frac{1}{z-1}}{j(\tau)-j(z)}.
\end{align*}
As $\Delta^{\ell}E_k(-1/(z-1))=(z-1)^{-k}\Delta^{\ell}E_k(z)$ and $j(-1/(z-1))=j(z)$, the residue is given by
\begin{align*}
    \text{Res}_{\tau=\frac{-1}{z-1}}H_1(\tau,z)=-\frac{1}{2\pi i}\ell E_2(z)(z-1)^{-k}e^{-2\pi im(-1/(z-1))}-\frac{1}{2\pi i}(z-1)^{-k}\frac{DE_{k'}(z)}{E_{k'}(z)}e^{-2\pi im(-1/(z-1))}
\end{align*}
Using $-1/(z-1)=2i(\sin \theta)e^{-i\theta}$ and $(z-1)^{-k}=i^{-k}|z|^{-k}e^{-ik\theta}$, we obtain
    
    \begin{align*}
        &\text{Res}_{\tau=\frac{-1}{z-1}}H_1(\tau,z)=-\frac{1}{2\pi i}\ell E_2(z)i^{-k}|  z|  ^{-k}e^{-ik\theta+4\pi m (\sin \theta) e^{-i\theta}}-\frac{1}{2\pi i}\frac{DE_{k'}(z)}{E_{k'}(z)}i^{-k}|  z|  ^{-k}e^{-ik\theta+4\pi m (\sin \theta) e^{-i\theta}}
    \end{align*}
Similarly, we can obtain
    \begin{align*}
        &\text{Res}_{\tau=\frac{-1}{z}}H_1(\tau,z)=-\frac{1}{2\pi i}\ell E_2(z)i^{-k}|  z|  ^{-k}e^{ik\theta+4\pi m (\sin \theta) e^{i\theta}}-\frac{1}{2\pi i}\frac{DE_{k'}(z)}{E_{k'}(z)}i^{-k}|  z|  ^{-k}e^{ik\theta+4\pi m (\sin \theta) e^{i\theta}}.
    \end{align*}

On the other hand, the function $H_2(\tau,z)$ has a double pole at $\tau=-1/(z-1)$, we compute the residue by
\begin{align*}
    \text{Res}_{\tau=\frac{-1}{z-1}}H_2(\tau,z)=\lim_{\tau \to -1/(z-1)}\frac{d}{d\tau}\left(\left(\tau+\frac{1}{z-1}\right)^2H_2(\tau,z)\right)
\end{align*}
Using the equation $\Delta'(\tau)=2\pi i \Delta(\tau)E_2(\tau)$ and the relations the values of modular forms at $\tau=-1/(z-1)$ and at $\tau=z$ as before, we get
    \begin{align*}
    \text{Res}_{\tau=\frac{-1}{z-1}}H_2(\tau,z)&=-\frac{i}{4\pi^2}e^{2\pi i m/(z-1)}(z-1)^{-k-2}\left(2\pi m + 2\pi \ell (z-1)^2 E_2(z)-ik(z-1)+2\pi (z-1)^2 \frac{DE_{k'}(z)}{E_{k'}(z)}\right).
    \end{align*}
Similarly, for $\tau=-1/z$, we have
    \begin{align*}
        \text{Res}_{\tau=\frac{-1}{z}}H_2(\tau,z)&=-\frac{i}{4\pi^2}e^{2\pi i m/z}z^{-k-2}\left(2\pi m + 2\pi \ell z^2 E_2(z)-ikz+2\pi z^2 \frac{DE_{k'}(z)}{E_{k'}(z)}\right).
    \end{align*}

    Observe that from the above equations for residues, one has 
    \begin{align*}
         \left(\text{Res}_{\tau=\frac{-1}{z}}+\text{Res}_{\tau=\frac{-1}{z-1}}\right)\left(H_1(\tau,z)+\frac{i}{4}\left(2\pi \ell E_2(z)+2\pi \frac{DE_{k'}(z)}{E_{k'}(z)}\right)\left(e^{2\pi i m/z}z^{-k}+e^{2\pi i m/(z-1)}(z-1)^{-k}\right)\right)=0.
    \end{align*}
    Therefore, calculating the remaining terms of residues leads to the desired result.
\end{proof}

We establish the following proposition using Lemma~\ref{lem_DJ08 lemma2} and Lemma~\ref{lem_H'' approximation}, which leads to a direct deduction of Theorem \ref{thm_gap ftn}(a). When combined with Proposition~\ref{prop_number of zeros of f_k,0}, this also deduces Theorem~\ref{thm_gap ftn}(b).

\begin{proposition}\label{prop_main thm for gap ftn}
    Let $k'=0$ and let $A''=0.49$. If $k \geq 1116$ and $0.38 \leq \theta < \pi/6$,
    then we have
    \begin{align*}
        \left|\int_{-\frac{1}{2}+iA''}^{\frac{1}{2}+iA''} H(\tau,z)d\tau\right|<\frac{k}{\pi}|  z|  ^{-k-1}.
    \end{align*}
    Consequently, if $k$ is large enough, the sign of $Df_{k,0}\left(\frac{1}{2}+it_j\right)$ is $(-1)^j$ for $19(k+1)/50\pi \leq j \leq [k/6]-1$.
\end{proposition}

\begin{proof}
    \www{}
    We have the inequality
    \begin{align*}
         |z|^{k+1}\left|\int_{-\frac{1}{2}+iA''}^{\frac{1}{2}+iA''} H(\tau,z)d\tau\right|&=(2\sin \theta)^{-k-1}\left|\int_{-\frac{1}{2}+iA''}^{\frac{1}{2}+iA''} H(\tau,z)d\tau\right| \\
         &\leq (2\sin \theta)^{-k-1}\max_{|  x|  \leq 1/2}|  H_1(x+iA'',z)|  +(2\sin \theta)^{-k-1}\max_{|  x|  \leq 1/2}|  H_2(x+iA'',z)|.
         \end{align*}
         Note that since $E_{k'}=1$,
         \begin{align*}
         &(2\sin \theta)^{-k-1}|  H_1(\tau,z)|  +(2\sin \theta)^{-k-1}|  H_2(\tau,z)|   \\
         &=(2\sin \theta)^{-1}\left|\frac{\Delta(z)}{(2\sin \theta)^{12}\Delta(\tau)}\right|^{\ell}\left(\left|\ell E_2(z) \frac{E_{14}(\tau)}{\Delta(\tau)(j(\tau)-j(z))}\right|+\left|\frac{E_{14}(\tau)E_{14}(z)}{\Delta(\tau)\Delta(z)(j(\tau)-j(z))^2} \right|\right).
    \end{align*}
    Recall $z=ie^{-i\theta}|z|=ie^{-i\theta}/(2\sin \theta)$ for $0.38 \leq \theta < \frac{[k/6]\pi}{k+1}$ and $\tau=x+0.49i$. The following bounds can be obtained through numerical computations:
    \begin{align*}
        \left|E_2(z)\right|<1.15, \quad\quad
        \left|\frac{E_{14}(\tau)E_{14}(z)}{\Delta(\tau)\Delta(z)(j(\tau)-j(z))^2} \right| <0.6, \\
    \left|\frac{E_{14}(\tau)}{\Delta(\tau)((j(\tau)-j(z)))} \right|<4.2, \quad \text{ and } \quad\left|\frac{\Delta(z)}{(2\sin \theta)^{12}\Delta(\tau)}\right|<0.99.
    \end{align*}
    By combining these bounds, we complete the proof.
\end{proof}

\begin{remark}\label{rmk_for other k'}
    Proposition \ref{prop_main thm for gap ftn} is specific to $k' = 0$, as we have noted in the introduction. To see this, consider the case where $k' = 14$. Using the same argument as in the proof of Theorem~\ref{thm_gap ftn}, we find that for $0.38 \leq \theta < \pi/6$ and for all $k=12\ell+14$,
\begin{align*}
|  z|  ^{k+1}Df_{k,0}(z)=-|  z|  ^{k+1}g_{14}(z)+R(z), \quad \text{ where } |  R(z)|  <\frac{2k}{\pi}.
\end{align*}
Recall that $$g_{14}(z)=E_2(z)(E_{14}(z)^2-1)\ell+g_{14,0}(z).$$ Numerical computations show that for $z=\frac{1}{2}+\frac{i}{2}\cot \theta=ie^{-i\theta}|  z|  $ with $\theta \in \left[\frac{(\ell+1)\pi}{k+1},\frac{([k/6]-1)\pi}{k+1}\right]$, $E_2(z)(E_{14}(z)^2-1)$ is negative and bounded away from $0$, and $g_{14,0}(z)$ is positive and bounded away from $0$ for $0.38\leq \theta < \pi/6$. As $\theta$ approaches $\pi/6$, both $|  E_{14}(z)^2-1|  $ and $g_{14,0}(z)$ tend to $\infty$. Note that for $\theta_0$ close enough to $\pi/6$ (e.g., $\theta_0=0.511$), if we take $0.38 \leq \theta \leq \theta_0$, then $|  z|  >\epsilon_0$ for some uniform constant $\epsilon_0>1$. Therefore, for large $k$, the function $|  z|  ^{k+1}Df_{k,0}(z)$ never vanishes.

However, the situation changes if we don't restrict the size of $\theta$. It is possible for the real zeros of $Df_{k,0}$ to occur at $z$ with $k' \neq 0$ and $\pi/12<\theta<0.38$. For example, if we take $k=86$, then there are $6$ zeros of $Df_{k,0}$ in $F$. We can verify numerically that there are $4$ real zeros and $2$ non-real zeros among them. Such real zeros have an imaginary part larger than $\frac{1}{2}+\frac{i}{2}\cot(0.38) \approx 1.2518$, but smaller than $\frac{1}{2}+\frac{i}{2}\cot(\pi/12) \approx 1.8660$.
\end{remark}

\section{other remarks on the zeros}\label{sec_other remarks}

In the previous sections, we have observed that certain types of quasimodular forms have zeros solely on the line $\delta_2$,   while other types have more than quarter of their zeros on this line, but not all of them. In this section, we explore the existence of real zeros.


Let $\mathfrak{f}$ be a cusp form such that $\mathfrak{f}'=f$. Note that by the valence formula, if $k\equiv 0 \pmod{4}$ then $i$ must be a zero of $\mathfrak{f}$. Since $\mathfrak{f}(i\infty)=0$, we have
\begin{align*}
    f(iy_0)=-i\frac{\partial \mathfrak{f}}{\partial y}(iy_0)=0
\end{align*}
for some $y_0 \in (1, \infty)$. This argument can be applied to the line $\{z \in F: \Re(z)=\frac{1}{2}\}$, since $\mathfrak{f}\left(\frac{1}{2}+\frac{i}{2}\right)=(i-1)^k \mathfrak{f}(i)$. \ww{These partially proved Theorem~\ref{thm_cusp form 2 crit pt}, and thus it is enough to consider $k \equiv 2 \pmod{4}$.}

For a quasimodular form $f$, let us denote by $a_{\infty}(f)$ the first non-zero Fourier coefficient of $f$. Also, if $f$ is non-cuspidal, we define $\epsilon_f'$ and $\epsilon_f$ as follows: Let $\epsilon_f'$ be the sign of the product of the first two non-zero Fourier coefficients of $f$, say $a_0$ and $a_n$,  and
$$\epsilon_f:=(-1)^{n}\epsilon_f'.$$

To prove Theorem \ref{thm_cusp form 2 crit pt} for the remaining case, we'll show a slightly more \ww{general statement} as follows.

\begin{proposition}\label{prop_crit.of.modularform.lying.onthe.central}
Let $f$ be a modular form of weight $k$ and assume that at least one of the following conditions holds:
\begin{enumerate}[\normalfont(i)]
    \item $f$ is cuspidal and $k \equiv 0 \pmod{4},$
    \item $f$ is non-cuspidal with $\epsilon_f=-1$ and $k \equiv 0 \pmod 4$,
    \item $f$ is non-cuspidal with $\epsilon_f=1$ and $k \equiv 2 \pmod 4.$
\end{enumerate}
Then $f$ has a critical point lying on the line $\{z \in \Ha: \Re(z)=1/2\}$. If we replace the conditions \textnormal{(ii)} and \textnormal{(iii)}, respectively by 
\begin{enumerate}[\normalfont(i)]
\setcounter{enumi}{3}
    \item $f$ is non-cuspidal with $\epsilon_f'=-1$ and $k \equiv 0 \pmod 4$,
    \item $f$ is non-cuspidal with $\epsilon_f'=1$ and $k \equiv 2 \pmod 4$,
\end{enumerate}
then $f$ has a critical point lying on the line $\{z \in \Ha: \Re(z)=0\}$.
\end{proposition}

To prove this, we need the following lemmas.

\begin{lemma}\label{lem_calculation of v^-ng/v^-mf}
Let $f$ and $g$ be quasimodular forms and let $a:=v_{\infty}(f)$ and $b:=v_{\infty}(g)$. If $a-b>0$, then
\begin{align*}
    \lim_{y \to \infty} \frac{y^{-n}g(iy)}{y^{-m}f(iy)}&=\sgn(a_{\infty}(f))\sgn(a_{\infty}(g))\infty, \\
    \lim_{y \to \infty} \frac{y^{-n}g(\frac{1}{2}+iy)}{y^{-m}f(\frac{1}{2}+iy)}&=(-1)^{a-b}\sgn(a_{\infty}(f))\sgn(a_{\infty}(g))\infty,
\end{align*}
and if $a-b<0$, then 
\begin{align*}
    \lim_{y \to \infty} \frac{y^{-n}g(iy)}{y^{-m}f(iy)}=\lim_{y \to \infty} \frac{y^{-n}g(\frac{1}{2}+iy)}{y^{-m}f(\frac{1}{2}+iy)}=0.
\end{align*}
Lastly if $a-b=0$, then 
\begin{align*}
    \lim_{y \to \infty} \frac{y^{-n}g(iy)}{y^{-m}f(iy)}=\lim_{y \to \infty} \frac{y^{-n}g(\frac{1}{2}+iy)}{y^{-m}f(\frac{1}{2}+iy)}=
    \begin{cases} 0 & \text{ if } n>m, \\ \frac{a_{\infty}(g)}{a_{\infty}(f)} & \text{ if } n=m, \\ \sgn(a_{\infty}(f))\sgn(a_{\infty}(g))\infty & \text{ if } n\le m. \end{cases}
\end{align*}
\end{lemma}

\begin{proof}
Let $x=1/2$ and $\hat{f}\circ q=f$, $\hat{g}\circ q=g$. Note that 
\begin{align*}
    \lim_{y \to \infty} \frac{y^{-n}g(x+iy)}{y^{-m}f(x+iy)}&=\lim_{y \to \infty} \frac{y^{-n}(-e^{2\pi y})^{-b}(-e^{-2\pi y})^{-b} g(x+iy)}{y^{-m}(-e^{2\pi y})^{-a}(-e^{-2\pi y})^{-a} f(x+iy)} \\
    &=(-1)^{a-b}\lim_{y \to \infty}\frac{\frac{e^{-2\pi by}}{y^n}q(x+iy)^{-b}\hat g (q(x+iy))}{\frac{e^{-2\pi ay}}{y^m}q(x+iy)^{-a}\hat f (q(x+iy))} \\
    &=(-1)^{a-b}\frac{a_{\infty}(g)}{a_{\infty}(f)}\lim_{y \to \infty} \frac{e^{2\pi y(a-b)}}{y^{n-m}}.
\end{align*}
Similarly if $x=0$, then
\begin{align*}
    \lim_{y \to \infty} \frac{y^{-n}g(x+iy)}{y^{-m}f(x+iy)}=\frac{a_{\infty}(g)}{a_{\infty}(f)}\lim_{y \to \infty} \frac{e^{2\pi y(a-b)}}{y^{n-m}}.
\end{align*}
The assertion follows from the above equations immediately.
\end{proof}

\begin{lemma}\label{lem_thetaf assymptotic formula, f:depth0}
Let $f$ be a modular form of weight $k$.
\begin{enumerate}[\normalfont(a)]
    \item If $f$ is cuspidal, then for any positive integer $j$ we have
        \begin{align*}
        \begin{cases}
        D^j f\left(iy\right)=\left(\frac{i}{y}\right)^{k+2j}D^j f\left(\frac{i}{y}\right)\times \left(1+o(1) \right), \\
        D^j f\left(\frac{1}{2}+iy\right)=\left(\frac{i}{2y}\right)^{k+2j}D^j f\left(\frac{1}{2}+\frac{i}{4y}\right)\times \left(1+o(1) \right),
        \end{cases}
        \end{align*}
        as $y$ approaches $0.$
    \item If $f$ is non-cuspidal, then
    \begin{align*}
    \begin{cases}
    D f(iy)=\left(\frac{i}{y}\right)^{k+2}D f\left(\frac{i}{y}\right)\times \left(1-\epsilon_f'\omega(1) \right), \\
    D f(\frac{1}{2}+iy)=\left(\frac{i}{2y}\right)^{k+2}D f\left(\frac{1}{2}+\frac{i}{4y}\right)\times \left(1-\epsilon_f\omega(1) \right),
    \end{cases}
    \end{align*}
    as $y$ approaches $0.$
    \item If $f$ is non-cuspidal, then for $j>1$,
    \begin{align*}
    \begin{cases}
    D^j f(iy)=\left(\frac{i}{y}\right)^{k+2j}D^j f\left(\frac{i}{y}\right)\times \left(1+(-1)^{j-1}\epsilon_f'\omega(1) \right), \\
    D^j f(\frac{1}{2}+iy)=\left(\frac{i}{2y}\right)^{k+2j}D^j f\left(\frac{1}{2}+\frac{i}{4y}\right)\times \left(1+(-1)^{j-1}\epsilon_f\omega(1) \right),
    \end{cases}
    \end{align*}
    as $y$ approaches $0$.
\end{enumerate}
Here, little-$o$ and little-$\omega$  are the asymptotic bounds.
\end{lemma}

\begin{proof}
First, we prove (a) for $j=1$ and (b). Since $D f$ is a quasimodular form of weight~$k+2$ and depth~$1$, we have
\begin{align*}
    D f(\gamma z)=(cz+d)^{k+2}D f(z)+\frac{ck}{2\pi i}(cz+d)^{k+1}f(z)
\end{align*}
for any $\gamma \in \SL_2(\Z)$ and $z \in \Ha.$
If we put $\gamma=\begin{pmatrix}1 & 0 \\ 2 & 1 \end{pmatrix}$ and $z=-\frac{1}{2}+\frac{i}{4y},$ then
\begin{align*}
    Df\left(\frac{1}{2}+iy\right)&=\left(\frac{i}{2y}\right)^{k+2}Df\left(-\frac{1}{2}+\frac{i}{4y}\right)+\left(\frac{i}{2y}\right)^{k+1}\cdot 2 \cdot \frac{k}{2\pi i}f\left(-\frac{1}{2}+\frac{i}{4y}\right) \\
    &=\left(\frac{i}{2y}\right)^{k+2}Df\left(\frac{1}{2}+\frac{i}{4y}\right)\times \left(1-\frac{k}{2\pi }\frac{f(\frac{1}{2}+\frac{i}{4y})}{\frac{1}{4y}D f(\frac{1}{2}+\frac{i}{4y})} \right).
\end{align*}
Recall that $D f$ is always cuspidal, so $v_{\infty}(Df)>0$. 

Suppose $f$ is cuspidal. Since $D =q\frac{d}{dq},$ it is clear that $v_{\infty}(f)=v_{\infty}(Df).$ By Lemma \ref{lem_calculation of v^-ng/v^-mf} we have
\begin{align*}
    \lim_{y \to 0}\frac{f(\frac{1}{2}+\frac{i}{4y})}{\frac{1}{4y}D f(\frac{1}{2}+\frac{i}{4y})}=0.
\end{align*}
This proves the second equation of (a) for $j=1$. The first equation of (a) for $j=1$ is obtained by applying the same argument for $\gamma=\begin{pmatrix} 0 & -1 \\ 1 & 0 \end{pmatrix}$, $z=\frac{i}{y}$. 


Suppose $f$ is non-cuspidal. Note that in this case, $a_{\infty}(f)a_{\infty}(D f)=na_0 a_n$, where $a_0$ and $a_n$ are the first two non-zero Fourier coefficients of $f$. According to Lemma~\ref{lem_calculation of v^-ng/v^-mf} again, we have
\begin{align*}
    \lim_{y \to 0}\frac{f(\frac{1}{2}+\frac{i}{4y})}{\frac{1}{4y}D f(\frac{1}{2}+\frac{i}{4y})}=(-1)^{v_{\infty}(Df)-v_{\infty}(f)}\sgn(a_{\infty}(f))\sgn(a_{\infty}(D f))\infty=\epsilon_f \infty,
\end{align*}
which implies the second equation of (b), and similarly the first equation is derived as well.

Next, we'll consider (a) for $j>1$ and (c). Among these, we only prove the second equation in (c), since the proof of (a) for $j>1$ and the first equation in (c) can be obtained by the similar argument. 

Recall that for an arbitrary quasimodular form $g$ of weight $k$ and depth $\ell$, we have

\begin{align*}
    g(\gamma z)=\sum_{m=0}^{\ell}c^j(cz+d)^{k-j}Q_m(g)(z)
\end{align*}
for arbitrary $\gamma \in \SL_2(\Z)$. Thus for any $\gamma~\in~\SL_2(\Z)$,
\begin{align*}
    D^j f(\gamma z)=\sum_{m=0}^j c^m(cz+d)^{k+2j-m}Q_m(D^j f)(z),
\end{align*}
and in particular
\begin{align}\label{eqn_theat j f quasimodularity, Re 1/2}
    D^j f\left(\frac{1}{2}+iy\right)&=\sum_{m=0}^j 2^m\left(\frac{i}{2y}\right)^{k+2j-m}Q_m(D^j f)\left(\frac{1}{2}+\frac{i}{4y}\right) \notag \\
    &=\left(\frac{i}{2y}\right)^{k+2j}D^j f\left(\frac{1}{2}+\frac{i}{4y}\right)\times \left(1+\sum_{m=1}^j 2^m\left(\frac{i}{2y}\right)^{-m}\frac{Q_m(D^j f)\left(\frac{1}{2}+\frac{i}{4y}\right)}{D^j f\left(\frac{1}{2}+\frac{i}{4y}\right)}\right).
\end{align}
Referring to \cite[Theorem 3.5]{Roy12}, it can be verified that $Q_0(D^j f)=D^j f$, $Q_j(D^j f)=\frac{j!}{(2\pi i)^j}{k\choose j}f$, and for $1 \leq m \leq j-1$, $Q_m(D^j f)$ is the $\frac{m!}{(2\pi i)^m} {k\choose m}$ times $\Z$-linear combinations of $D^{j-m-1}f$ and $D^{j-m}f$. Thus for $m<j-1$, a function $Q_m(D^j f)$ is cuspidal. Furthermore, $Q_{j-1}(D^j f)$ can be shown inductively to be equal to $\frac{(j-1)!}{(2\pi i)^{j-1}} {k\choose j-1}(f+(j-1)D f)$. 

Therefore, by applying Lemma \ref{lem_calculation of v^-ng/v^-mf}, we conclude that the equation \eqref{eqn_theat j f quasimodularity, Re 1/2} can be expressed as
\begin{align*}
    &\left(\frac{i}{2y}\right)^{k+2j}D^j f\left(\frac{1}{2}+\frac{i}{4y}\right)\times \left(1+\sum_{m=j-1}^j 2^m\left(\frac{i}{2y}\right)^{-m}\frac{Q_m(D^j f)\left(\frac{1}{2}+\frac{i}{4y}\right)}{D^j f\left(\frac{1}{2}+\frac{i}{4y}\right)}+o(1)\right) \\
    &=\left(\frac{i}{2y}\right)^{k+2j}D^j f\left(\frac{1}{2}+\frac{i}{4y}\right)\times \left(1+\left(\frac{(j-1)!}{i^{j-1}(2\pi i)^{j-1}}{k\choose j-1}+\frac{j!}{i^j(2\pi i)^j}{k\choose j}\cdot 4y\right)\frac{f\left(\frac{1}{2}+\frac{i}{4y}\right)}{\left(\frac{1}{4y}\right)^{j-1}D^j f\left(\frac{1}{2}+\frac{i}{4y}\right)}+o(1) \right) \\
    &=\left(\frac{i}{2y}\right)^{k+2j}D^j f\left(\frac{1}{2}+\frac{i}{4y}\right)\times (1+(-1)^{j-1}\epsilon_f \omega(1)),
\end{align*}
as $y$ approaches $0$.
\end{proof}

\begin{proof}[Proof of Proposition \ref{prop_crit.of.modularform.lying.onthe.central}]
 As $f$ has real Fourier coefficients, so does its derivative $Df$. If $f$ is cuspidal (resp. non-cuspidal), by Lemma \ref{lem_thetaf assymptotic formula, f:depth0}, we have $D f(\frac{1}{2}+iy)=\left(\frac{i}{2y}\right)^{k+2} D f(\frac{1}{2}+\frac{i}{4y})\times (1+o(1))$ (resp. $\left(\frac{i}{2y}\right)^{k+2}D f(\frac{1}{2}+\frac{i}{4y})\times \left(1-\epsilon_f\omega(1) \right)$), so $D f(\frac{1}{2}+iy)$ and $D f(\frac{1}{2}+\frac{i}{4y})$ are real numbers with opposite signs for some $y \gg 0.$

It immediately follows that there exists $y_0 \in (\frac{1}{4y},y)$ such that $Df(\frac{1}{2}+iy_0)=0$. Similarly, one can verify the assertion for $\{z \in \mathbb H: \Re(z)=0\}$.
\end{proof}

Unless $f$ belongs to $DS_{k-2,\R}$, the existence of a real zero holds under certain weight conditions. Before presenting this, we introduce the following lemma.

\begin{lemma}
Let $f$ be a cuspidal quasimodular form. As $y \gg 0$, the sign of $f\left(\frac{1}{2}+iy\right)$ is $(-1)^{v_{\infty}(f)}\sgn(a_{\infty}(f))$, and the sign of $f(iy)$ is $\sgn(a_{\infty}(f))$. 
\end{lemma}

\begin{proof}
It follows from $f(z)=\sum_{n=v_{\infty}(f)}^{\infty}a_n q^n=q^{v_{\infty}(f)}\left(a_{v_{\infty}(f)}+O(q)\right)$ immediately.
\end{proof}

\begin{proposition}
Let $f$ be a weight $k$ quasimodular form. 
\begin{enumerate}[\normalfont(a)]
    \item If $k \equiv 2 \pmod{4}$ and $f$ is non-cuspidal, then $f$ has at least two zeros on $\{z \in \Ha: \Re(z)=0 \text{ or } 1/2\}$, one of them lying on $\{z \in \Ha: \Re(z)=0\}$ and one of them on $\{z \in \Ha: \Re(z)= 1/2\}$.
    \item If $k \equiv 2 \pmod{4}$ and $f=f_0+Df_1 \in S_{k,\R} \oplus D  M_{k-2,\R}$ with $v_{\infty}(f) \leq v_{\infty}(f_1)$, then $f$ has a zero on the line~$\{z \in \Ha: \Re(z)= 1/2\}$.
    \item If the depth of $f$ is not larger than $1$ and $k \equiv 6,10 \pmod{12}$, then $f$ has a zero on the line~$\{z \in \Ha: \Re(z)= 1/2\}$.
    \item If the depth of $f$ is not larger than $2$ and $k \equiv 6 \pmod{12}$, then $f$ has a zero on the line $\{z \in \Ha: \Re(z)= 1/2\}$.
\end{enumerate}
\end{proposition}

\begin{proof}
Let $\ell$ be the depth of $f$. Recall that
\begin{align*}
    f\left(\frac{1}{2}+iy\right)=\left(\frac{i}{2y}\right)^k f\left(\frac{1}{2}+\frac{i}{4y}\right) \times \left(1+\sum_{j=1}^{\ell} \frac{1}{i^j}\frac{Q_j(f)}{\left(\frac{1}{4y}\right)^j f}\left(\frac{1}{2}+\frac{i}{4y}\right) \right),
\end{align*}
and similarly,
\begin{align*}
     f\left(iy\right)=\left(\frac{i}{y}\right)^k f\left(\frac{i}{y}\right) \times \left(1+\sum_{j=1}^{\ell} \frac{1}{i^j}\frac{Q_j(f)}{\left(\frac{1}{y}\right)^j f}\left(\frac{i}{y}\right) \right).
\end{align*}
Since $f$ is non-cuspidal, we have $0=v_{\infty}(f) \leq v_{\infty}(Q_j(f))$ for any $j \in \{1,2,\ldots,\ell\}$ so that both of  $\frac{Q_j(f)}{\left(\frac{1}{4y}\right)^j f}\left(\frac{1}{2}+\frac{i}{4y}\right)$ and $\frac{Q_j(f)}{\left(\frac{1}{y}\right)^j f}\left(\frac{i}{y}\right)$ are $o(1)$ as $y$ approaches $0$ by Lemma \ref{lem_calculation of v^-ng/v^-mf}. In particular, the signs of $f\left(\frac{1}{2}+\frac{i}{4y}\right)$ and of $f\left(\frac{i}{y}\right)$ are opposite to the sign of $f(\infty)$. This completes the proof of~(a), and the same arguments can be applied to prove (b).

To prove (c), write $f=f_0+D f_1$ for $f_0 \in M_{k,\R}$, $f_1 \in M_{k-2,\R}$. We have
\begin{align*}
    f\left(\frac{1}{2}+iy\right)=\left(\frac{i}{2y}\right)^{k} f\left(\frac{1}{2}+\frac{i}{4y}\right)\times \left(1-\frac{k}{2\pi}\frac{f_1}{\frac{1}{4y}f}\left(\frac{1}{2}+\frac{i}{4y}\right)\right).
\end{align*}
Since $k-2 \equiv 4,8 \pmod{12}$, we have $f_1\left(\frac{1}{2}+i\frac{\sqrt 3}{2}\right)=0$ by the valence formula. If we take $y=\frac{i}{2\sqrt3}$, then we get
\begin{align*}
    f\left(\frac{1}{2}+\frac{i}{2\sqrt 3}\right)=-\left(\frac{5}{2}\right)^k f\left(\frac{1}{2}+i\frac{\sqrt 3}{2}\right),
\end{align*}
hence the desired result follows. A similar argument also proves (d).
\end{proof}

It is natural to question whether there is a quasimodular form with no real zeros on lines $\{z \in \mathbb H : \Re(z)=0\}$ or $\{z \in \mathbb H : \Re(z)=1/2\}$, particularly when the depth is greater than $0$. In the case of depth $0$, the modular discriminant function $\Delta$ is one of the standard examples of such (quasi)modular forms, but the answer is unclear for higher depths.

For quasimodular forms of depth $1$, it is clear that the derivative of the Eisenstein series does not vanish on the line $\{z \in \Ha: \Re(z)=0\}$. Similarly, we will prove that there exist quasimodular forms of depth $1$ that do not vanish on the line $\{z \in \Ha: \Re(z)=1/2\}$. Remarkably, these forms are the derivatives of modular forms, and the corresponding antiderivative functions also lack real zeros on the line~$\{z \in \Ha: \Re(z)=1/2\}$.

\begin{proposition}\label{thm_monotonedecreasing noncuspidal}
For $k \equiv 0 \pmod{12}$, let $f$ be a non-cuspidal modular form of weight $k$ given by
\begin{align*}
    f(z)=b_1 \Delta^{\frac{k}{12}}(z)+b_k\Delta(z)E_4^{\frac{k}{4}-3}(z)+E_4^{\frac{k}{4}}(z).
\end{align*}
For sufficiently large $k$, there exists a constant $B>0$ depending on $b_k$ such that if $(-1)^{\frac{k}{12}}b_1>B$ then $f$ has no zero lying on the line $\{z \in \Ha: \Re(z)= 1/2\} \cup \{\infty\}$. Furthermore, if we assume $b_k<-60k$, then we can take $B$ for which $f\left(\frac{1}{2}+iy\right)$ is a monotone decreasing function for $y \in (0,\infty)$. 
\end{proposition}

To prove Proposition \ref{thm_monotonedecreasing noncuspidal}, we first investigate the sign changes of the Ramanujan tau function~$\tau(n)$ and prove several useful lemmas. Define
\begin{align*}
    \tau_{m}(n):=\sum_{\substack{a_1,a_2,\ldots, a_m\geq 1 \\ a_1+a_2+\cdots+a_m=n }}\tau(a_1)\tau(a_2)\cdots\tau(a_m),
\end{align*}
be the $n$th Fourier coefficient of $\Delta^m$. For a positive integer $k \equiv 0 \pmod{12}$, let  $N_k:=\frac{k}{12}+\left[1+\frac{2}{15}k\right]$. We now introduce the D'Arcais polynomial $P_n(x)$, which is defined recursively by
\begin{align*}
\begin{cases}
P_0(x)=1, \\
P_n(x)=\frac{x}{n}\sum_{j=1}^n \sigma_1(j)P_{n-j}(x), & \text{ for } n \geq 1.
\end{cases}
\end{align*}
It is well-known that all real roots of $P_n(x)$ are negative. Additionally, the special values of $P_n(x)$ correspond to the $q$-coefficients of the power of the Dedekind eta function, namely, if we write
\begin{align*}
\prod_{n\geq 1}(1-q^n)^r=\sum_{n=0}^{\infty}\eta_n(r)q^n,
\end{align*}
for $r \in \C$, then we have $P_n(-r)=\eta_n(r)$ (see \cite{D'ar13}, \cite{New55}). 

Furthermore, the sign of each $\eta_n(r)$ is determined by the D’arcais polynomial in the following manner.
\begin{theorem}{\cite[Theorem 2]{HN20(a)}}\label{thm_root of Darcais}
Let $r_n$ be the number of real roots of $P_n(x)$ which are less than or equal to $-r$. Then we have
$
    (-1)^{n+r_n}\eta_n(r) \geq 0.
$
\end{theorem}
By virtue of Theorem \ref{thm_root of Darcais}, we may reach an immediate conclusion that for $r=2k$, should any real root of $P_n(x)$ exceed $-2k$, then we have $(-1)^n \tau_{\frac{k}{12}}\left(n+\frac{k}{12}\right)<0$.

Recently, Heim and Neuhauser \cite{HN20(b)} established the growth condition for $P_n(x)$, which constitutes an improvement of the results presented in \cite{Kos04} and \cite{Han10}.

\begin{theorem}{\cite{HN20(b)}}\label{thm_Pnx inequality}
If $|  x|  >15(n-1)$, then $P_n(x) \neq 0$.
\end{theorem}

Combining Theorem \ref{thm_root of Darcais} and Theorem~\ref{thm_Pnx inequality}, we get the following lemma.

\begin{lemma}\label{lem_Ramanujantau sign}
Let $n$ be a positive integer. If $\frac{k}{12} \leq n \leq N_k$, then $\tau_{\frac{k}{12}}(n-1)\tau_{\frac{k}{12}}(n)<0$.
\end{lemma}

\begin{proof}
For $n=N_k-\frac{k}{12}$, we have $n<1+\frac{2k}{15}$, which is equivalent to $-2k<-15(n-1)$. This implies that there is no real root of $P_n(x)$ less than or equal to $-2k$. Therefore, for $0 \leq n \leq N_k-\frac{k}{12}$, we have $(-1)^n\tau_{\frac{k}{12}}\left(n+\frac{k}{12}\right)<0$.
\end{proof}

\begin{proof}[Proof of Proposition \ref{thm_monotonedecreasing noncuspidal}]
Fix a sufficiently large $k$. We'll specify later how large $k$ should be. Take into account the function $f/b_1$ by splitting into three partial sums denoted by $f_1, f_2,$ and $f_3$ as follows;
\begin{align*}
    \frac{1}{b_1}f(\tau)=\left(\sum_{n=0}^{\frac{k}{12}-1}+\sum_{n=\frac{k}{12}}^{N_k}+\sum_{n=N_k+1}^{\infty}\right)\left(\frac{a_n(f)}{b_1}q^n\right)=:f_1(z)+f_2(z)+f_3(z).
\end{align*}
Observing that $f_1$ is a finite sum, we have $f_1=O(|  b_1|  ^{-1})$ for any complex number $q$ with $0<|  q|  <1$. Furthermore, for any $n\geq \frac{k}{12}$, we can express $\frac{a_n(f)}{b_1}$ as $\tau_{\frac{k}{12}}(n)+\frac{1}{b_1}g(n)$, where $g(n)$ is given by
\begin{align*}
g(n)=\sum_{\substack{c_1,\ldots,c_{\frac{k}{4}}\geq 0 \\ c_1+\cdots+c_{\frac{k}{4}}=n }}240^{\frac{k}{4}}\sigma_3(c_1)\cdots \sigma_3(c_\frac{k}{4})+b_k \sum_{\substack{a>0, a_1,\ldots,a_{\frac{k}{4}-3}\geq 0 \\ a+a_1+\cdots+a_{\frac{k}{4}-3}=n }}240^{\frac{k}{4}-3}\tau(a)\sigma_3(a_1)\cdots\sigma_3(a_{\frac{k}{4}-3}).
\end{align*}
Here, we adopt the convention that $\sigma_3(0):=\frac{1}{240}$. 

To estimate the functions $\tau_{\frac{k}{12}}(n)$ and $g(n)$, we use the Ramanujan-Petersson bound $|  \tau(n)|~\leq~d(n)n^{\frac{11}{2}}$. Specifically, we have
\begin{align*}
    |  \tau_{\frac{k}{12}}(n)|  &\leq \sum_{\substack{c_1,\ldots,c_{\frac{k}{12}}\geq 0 \\ c_1+\cdots+c_{\frac{k}{12}}=n }}|  \tau(c_1)\cdots\tau(c_{\frac{k}{12}})|   \\
    & \leq \sum_{\substack{c_1,\ldots,c_{\frac{k}{12}}\geq 0 \\ c_1+\cdots+c_{\frac{k}{12}}=n}}d(c_1)\cdots d(c_{\frac{k}{12}})(c_1\cdots c_{\frac{k}{12}})^{\frac{11}{2}} \\
    &\leq \sum_{\substack{c_1,\ldots,c_{\frac{k}{12}}\geq 0 \\ c_1+\cdots+c_{\frac{k}{12}}=n}} (c_1\cdots c_{\frac{k}{12}})^{\frac{k+1}{2}} \leq p\left(n-\frac{k}{12}\right) \left(\frac{12n}{k}\right)^{\frac{13k}{24}},
\end{align*}
where $p(n)$ is the partition number of $n$. Similarly for $g(n)$, we have
\begin{align*}
    |  g(n)|   &\leq \sum_{\substack{c_1,\ldots,c_{\frac{k}{12}}\geq 0 \\ c_1+\cdots+c_{\frac{k}{4}}=n}} 240^{\frac{k}{4}}(c_1\cdots c_{\frac{k}{4}})^4+|  b_k|   \sum_{\substack{a>0, a_1,\ldots,a_{\frac{k}{4}-3}\geq 0 \\ a+a_1+\cdots+a_{\frac{k}{4}-3}=n }}240^{\frac{k}{4}-3}d(a)a^{\frac{k-1}{2}}(a_1\cdots a_{\frac{k}{4}-3})^4 \\
    &\leq 240^{\frac{k}{4}}p(n)\left(\frac{4n}{k}\right)^{\frac{k}{4}}\left(1+240^{-3}|  b_k|  n^{\frac{k-7}{2}}\right).
\end{align*}
By applying the well-known  estimate $p(n) \sim \frac{1}{4\sqrt{3} n^{3/2}}e^{\pi \sqrt{\frac{2}{3}n}}$ as $n\to\infty$ by Hardy and Ramanujan, we choose a sufficiently large positive number $B$ such that if  $|  b_1 |  >B$, then $$\frac{1}{|  b_1|  }\sum_{n=N_k+1}^{\infty} 240^{\frac{k}{4}}p(n)\left(\frac{4n}{k}\right)^{\frac{k}{4}}\left|1+240^{-3}|  b_k|  n^{\frac{k-7}{2}}\right|<\epsilon.$$ 
Hence, for $z=\frac{1}{2}+iy$ with $y\geq \frac{1}{2}$, we have 
\begin{align*}
    |  f_3(z)|  &\leq \sum_{n=N_k+1}^{\infty}|  \tau_{\frac{k}{12}}(n)|  |  q| ^n+\epsilon \\
    &\leq \sum_{n=N_k+1}^{\infty}p\left(n-\frac{k}{12}\right)\left(\frac{12n}{k}\right)^{\frac{k(k+1)}{24}}|  q| ^n+\epsilon \\
    & \leq \sum_{n=N_k+1}^{\infty}\frac{C}{4\left(n-\frac{k}{12}\right)\sqrt 3}e^{\pi \sqrt{\frac{2}{3}\left(n-\frac{k}{12}\right)}-2\pi y n}\left(\frac{12n}{k}\right)^{\frac{13k}{24}}+\epsilon
\end{align*}
for some constant $C>0$. Note that the last summation in the above inequality converges. We can ensure by choosing a sufficiently large $k$ that
\begin{align*}
    \sum_{n=N_k+1}^{\infty}\frac{C}{4\left(n-\frac{k}{12}\right)\sqrt 3}e^{\pi \sqrt{\frac{2}{3}\left(n-\frac{k}{12}\right)}-\pi n}\left(\frac{12n}{k}\right)^{\frac{13k}{24}}+\epsilon<2\epsilon.
\end{align*}
Also, we have $|  f_1(z)|  <\epsilon$ by taking large $B$.

We have found that the primary component in evaluating $f(z)$ is $f_2(z)$. As $|  b_1|  $ becomes large, the sign of $a_n(f)$ for $\frac{k}{12}\leq n \leq N_k$ is determined by the sign of $\tau_{\frac{K}{12}}(n)$, more precisely,
\begin{align*}   \sgn\left(a_n(f)\right)=\sgn\left(\tau_{\frac{k}{12}}(n)\right).
\end{align*}
Thus, according to Lemma \ref{lem_Ramanujantau sign}, the sign of $a_n(f)$ is $(-1)^{n-\frac{k}{12}}$ for $\frac{k}{12} \leq n \leq N_k$.

We choose a positive real number $y_0$ such that  if $y>y_0$ then $\left| f\left(\frac{1}{2}+iy\right)-1\right|<\epsilon$. Then for $z=\frac{1}{2}+iy$ with $y>y_0$, we have
\begin{align*}
    f_2(z)=\sum_{n=\frac{k}{12}}^{N_k}\frac{1}{b_1}a_n(f)(e^{-2\pi y})^n
    > \sum_{n=\frac{k}{12}}^{N_k}\frac{1}{b_1}(-1)^n a_n(f) e^{2\pi y_0 n}.
\end{align*}
Since we have chosen $b_1$ such that $(-1)^{\frac{k}{12}}b_1>B>0$, each term  $\frac{1}{b_1}(-1)^n a_n(f) e^{2\pi y_0 n}$ in the above summation is positive. If we let $M:=\sum_{n=\frac{k}{12}}^{N_k}\frac{1}{b_1}(-1)^n a_n(f) e^{2\pi y_0 n}$, then
\begin{align*}
    M=\sum_{n=\frac{k}{12}}^{N_k}(-1)^n\left( \tau_{\frac{k}{12}}(n)+\frac{1}{b_1}g(n)\right)e^{2\pi y_0 n} > e^{2\pi y_0}, \quad \text{as } |  b_1|  >B.
\end{align*}
Therefore, we conclude that
\begin{align*}
    \begin{cases}
    b_1f\left(\frac{1}{2}+iy\right)>M-3\epsilon & \text{if } \frac{1}{2} \leq y \leq y_0, \\
    b_1f\left(\frac{1}{2}+iy\right)>1-\epsilon & \text{if } y > y_0,
    \end{cases}
\end{align*}
which implies that $f\left(\frac{1}{2}+iy\right)$ is non-zero for any $y \in (0,\infty)$. 

It only remains to prove that if $b_k < -60k$, then $f\left(\frac{1}{2}+iy\right)$ is a monotone decreasing function for $y \in (0,\infty)$. Since $f$ is holomorphic, it suffices to show that $Df\left(\frac{1}{2}+iy\right) > 0$ for all $y \in (0,\infty)$. Note that $b_k < -60k$ implies that $v_{\infty}(Df) = 1$ and $\epsilon_f = 1$. Therefore, there exists $y_0 > 1/2$ such that if $y > y_0$, then $Df\left(\frac{1}{2}+iy\right) > 0$.

Recall Lemma \ref{lem_thetaf assymptotic formula, f:depth0}(b), which implies the existence of $y_1$ with $0 < y_1 < 1/2$ such that $Df\left(\frac{1}{2}+iy\right) > 0$ for $0 < y < y_1$ in this case. Thus, it suffices to show $Df\left(\frac{1}{2}+iy\right) > 0$ for $y \in [y_1,y_0]$. Since $Df = \sum_{n =1}^{\infty} na_n(f)q^n$, this follows by the same argument we used to prove the positivity of $f\left(\frac{1}{2}+iy\right)$.
\end{proof}

\begin{remark}
The given weight $k$ in Proposition \ref{thm_monotonedecreasing noncuspidal} does not need to be excessively large. In fact, it is enough to take $k \geq 24$. One example of such a modular form $f$ of weight $24$ is given by
\begin{align*}
f(z)&:=2\cdot 1728 \cdot 2880\Delta(z)^2-2880\Delta(z)E_4(z)^3+E_4(z)^6 \\
&=9953280 \Delta(z)^2-\frac{3}{225\zeta(4)^3}g_2(z)^3+\frac{1}{216000\zeta(4)^6}g_2(z)^6.
\end{align*}
\end{remark}

\begin{remark}
Proposition \ref{thm_monotonedecreasing noncuspidal} provides that if $k \equiv 0 \pmod{12}$, there exist infinitely many non-cuspidal modular forms of weight $k$ and quasimodular forms of weight $k+2$ and depth~$1$ which do not vanish on the line $\{z \in \Ha: \Re(z)=1/2\}$. However, it is not possible to construct such forms for depth~$\geq 2$ in the same manner as for $f$ and $Df$. In fact, for $j>1$, $D^j f$ always has a zero on the line $\{z \in \Ha: \Re(z)=1/2\}$, since the signs of $D ^j f\left(\frac{1}{2}+iy\right)$ and $D^j f\left(\frac{1}{2}+\frac{i}{4y}\right)$ are opposite for small~$y$, as shown in Lemma~\ref{lem_thetaf assymptotic formula, f:depth0}.
\end{remark}

\end{document}